\documentclass[aos]{imsart}

\usepackage{amsmath, amsthm, amssymb, amscd, paralist}
\usepackage{graphicx}
\usepackage{bm}
\usepackage[title]{appendix}
\usepackage{hyperref}

\arxiv{1708.06633}

\startlocaldefs

\newcommand{\eps}{\varepsilon}

\newcommand{\argmin}{\mathop{\rm arg\min}}

\newtheorem{thm}{Theorem}
\newtheorem{lem}{Lemma}
\newtheorem{rem}{Remark}

\newtheorem{cor}{Corollary}

\newtheorem{lem2}{Lemma}[section]

\newcommand{\Var}{\operatorname{Var}}

\newcommand{\Unif}{\operatorname{Unif}}

\newcommand{\ERM}{\operatorname{ERM}}

\newcommand{\Ham}{\operatorname{Ham}}

\newcommand{\Mult}{\operatorname{Mult}}
\newcommand{\Mon}{\operatorname{Mon}}
\renewcommand{\Hat}{\operatorname{Hat}}

\newcommand{\KL}{\operatorname{KL}}

\newcommand{\R}{\mathbb{R}}

\newcommand{\bt}{\mathbf{t}}
\newcommand{\ba}{\mathbf{a}}

\newcommand{\bd}{\mathbf{d}}

\newcommand{\mC}{\mathcal{C}}
\newcommand{\mD}{\mathcal{D}}

\newcommand{\mF}{\mathcal{F}}
\newcommand{\mG}{\mathcal{G}}

\newcommand{\mN}{\mathcal{N}}

\newcommand{\mW}{\mathcal{W}}

\renewcommand{\ba}{\mathbf{a}}
\newcommand{\bD}{\mathbf{D}}

\newcommand{\bp}{\mathbf{p}}
\newcommand{\bq}{\mathbf{q}}
\newcommand{\bu}{\mathbf{u}}
\newcommand{\bv}{\mathbf{v}}
\newcommand{\bw}{\mathbf{w}}
\newcommand{\bx}{\mathbf{x}}
\newcommand{\bX}{\mathbf{X}}
\newcommand{\by}{\mathbf{y}}

\newcommand{\balpha}{\bm{\alpha}}
\newcommand{\bbeta}{\bm{\beta}}
\newcommand{\bgamma}{\bm{\gamma}}
\newcommand{\blambda}{\bm{\lambda}}
\newcommand{\bLambda}{\bm{\Lambda}}
\newcommand{\bell}{\bm{\ell}}

\endlocaldefs

\begin{document}

\begin{frontmatter}

\title{Nonparametric regression using deep neural networks with ReLU activation function}
\runtitle{Nonparametric regression using ReLU networks}


\author{\fnms{Johannes} \snm{Schmidt-Hieber}\ead[label=e1]{a.j.schmidt-hieber@utwente.nl}}
\address{University of Twente,\\ P.O. Box 217, \\ 7500 AE Enschede \\ The Netherlands\\ \printead{e1}}
\affiliation{University of Twente}

\runauthor{J. Schmidt-Hieber}

\begin{abstract}
Consider the multivariate nonparametric regression model. It is shown that estimators based on sparsely connected deep neural networks with ReLU activation function and properly chosen network architecture achieve the minimax rates of convergence (up to $\log n$-factors) under a general composition assumption on the regression function. The framework includes many well-studied structural constraints such as (generalized) additive models. While there is a lot of flexibility in the network architecture, the tuning parameter is the sparsity of the network. Specifically, we consider large networks with number of potential network parameters exceeding the sample size. The analysis gives some insights into why multilayer feedforward neural networks perform well in practice. Interestingly, for ReLU activation function the depth (number of layers) of the neural network architectures plays an important role and our theory suggests that for nonparametric regression, scaling the network depth with the sample size is natural. It is also shown that under the composition assumption wavelet estimators can only achieve suboptimal rates.
\end{abstract}



\end{frontmatter}

\section{Introduction}
In the nonparametric regression model with random covariates in the unit hypercube, we observe $n$ i.i.d. vectors $\bX_i \in [0,1]^{d}$ and $n$ responses $Y_i \in \mathbb{R}$ from the model  
\begin{align}
	Y_i =f_0(\bX_i) + \eps_i, \quad i =1, \ldots, n.	
	\label{eq.mod}
\end{align}
The noise variables $\eps_i$ are assumed to be i.i.d. standard normal and independent of $(\bX_i)_i.$ The statistical problem is to recover the unknown function $f_0: [0,1]^d \rightarrow \mathbb{R}$ from the sample $(\bX_i,Y_i)_i.$ Various methods exist that allow to estimate the regression function nonparametrically, including kernel smoothing, series estimators/wavelets and splines, cf. \cite{gyorfi2002, wasserman2006, tsybakov2009}. In this work, we consider fitting a multilayer feedforward artificial neural network to the data. It is shown that the estimator achieves nearly optimal convergence rates under various constraints on the regression function. 

Multilayer (or deep) neural networks have been successfully trained recently to achieve impressive results for complicated tasks such as object detection on images and speech recognition. Deep learning is now considered to be the state-of-the art for these tasks. But there is a lack of mathematical understanding. One problem is that fitting a neural network to data is highly non-linear in the parameters. Moreover,  the function class is non-convex and different regularization methods are combined in practice.

This article is inspired by the idea to build a statistical theory that provides some understanding of these procedures. As the full method is too complex to be theoretically tractable, we need to make some selection of important characteristics that we believe are crucial for the success of the procedure. 

To fit a neural network, an activation function $\sigma : \mathbb{R}\rightarrow \mathbb{R}$ needs to be chosen. Traditionally, sigmoidal activation functions (differentiable functions that are bounded and monotonically increasing) were employed. For deep neural networks, however, there is a computational advantage using the non-sigmoidal rectified linear unit (ReLU) $\sigma(x)=\max(x,0)=(x)_+.$ In terms of statistical performance, the ReLU outperforms sigmoidal activation functions for classification problems \cite{glorot2011, Pedamonti2018} but for regression this remains unclear, see \cite{bauer2017}, Supplement B. Whereas earlier statistical work focuses mainly on shallow networks with sigmoidal activation functions, we provide statistical theory specifically for deep ReLU networks. 

The statistical analysis for the ReLU activation function is quite different from earlier approaches and we discuss this in more detail in the overview on related literature in Section \ref{sec.history}. Viewed as a nonparametric method, ReLU networks have some surprising properties. To explain this, notice that deep networks with ReLU activation produce functions that are piecewise linear in the input. Nonparametric methods which are based on piecewise linear approximations are typically not able to capture higher-order smoothness in the signal and are rate-optimal only up to smoothness index two. Interestingly, ReLU activation combined with a deep network architecture achieves  near minimax rates for arbitrary smoothness of the regression function. 
  
The number of hidden layers of state-of-the-art network architectures has been growing over the past years, cf. \cite{szegedy2016}. There are versions of the recently developed deep network ResNet which are based on $152$ layers, cf. \cite{he2016}. Our analysis indicates that for the ReLU activation function the network depth should be scaled with the sample size. This suggests, that for larger samples, additional hidden layers should be added. 

Recent deep architectures include more network parameters than training samples. The well-known AlexNet \cite{krizhevsky2012} for instance is based on $60$ million network parameters using only 1.2 million samples. We account for high-dimensional parameter spaces in our analysis by assuming that the number of potential network parameters is much larger than the sample size. For noisy data generated from the nonparametric regression model, overfitting does not lead to good generalization errors and incorporating regularization or sparsity in the estimator becomes essential. In the deep networks literature, one option is to make the network thinner assuming that only few parameters are non-zero (or active), cf. \cite{Goodfellow-et-al-2016-Book}, Section 7.10. Our analysis shows that the number of non-zero parameters plays the role of the effective model dimension and - as is common in non-parametric regression - needs to be chosen carefully.

Existing statistical theory often requires that the size of the network parameters tends to infinity as the sample size increases. In practice, estimated network weights are, however, rather small. We can incorporate small parameters in our theory, proving that it is sufficient to consider neural networks with all network parameters bounded in absolute value by one. 

Multilayer neural networks are typically applied to high-dimensional input. Without additional structure in the signal besides smoothness, nonparametric estimation rates are then slow because of the well-known curse of dimensionality. This means that no statistical procedure can do well regarding pointwise reconstruction of the signal. Multilayer neural networks are believed to be able to adapt to many different structures in the signal, therefore avoiding the curse of dimensionality and achieving faster rates in many situations. In this work, we stick to the regression setup and show that deep ReLU networks can indeed attain faster rates under a hierarchical composition assumption on the regression function, which includes (generalized) additive models and the composition models considered in \cite{horowitz2007, juditsky2009, baraud2014, kohler2017, bauer2017}.

Parts of the success of multilayer neural networks can be explained by the fast algorithms that are available to estimate the network weights from data. These iterative algorithms are based on minimization of some empirical loss function using stochastic gradient descent. Because of the non-convex function space, gradient descent methods might get stuck in a saddle point or converge to one of the potentially many local minima. \cite{choromanska2015} derives a heuristic argument showing that the risk of most of the local minima is not much larger than the risk of the global minimum. Despite the huge number of variations of the stochastic gradient descent, the common objective of all approaches is to reduce the empirical loss. In our framework we associate to any network reconstruction method a parameter quantifying the expected discrepancy between the achieved empirical risk and the global minimum of the energy landscape. The main theorem then states that a network estimator is minimax rate optimal (up to log factors) if and only if the method almost minimizes the empirical risk. 

We also show that wavelet series estimators are unable to adapt to the underlying structure under the composition assumption on the regression function. By deriving lower bounds, it is shown that the rates are suboptimal by a polynomial factor in the sample size $n.$ This provides an example of a function class for which fitting a neural network outperforms wavelet series estimators.

Our setting deviates in two aspects from the computer science literature on deep learning. Firstly, we consider regression and not classification. Secondly, we restrict ourselves in this article to multilayer feedforward artificial neural networks, while most of the many recent deep learning applications have been obtained using specific types of networks such as convolutional or recurrent neural networks. 

The article is structured as follows. Section \ref{sec.multilayer_NN} introduces multilayer feedforward artificial neural networks and discusses mathematical modeling. This section also contains the definition of the network classes. The considered function classes for the regression function and the main result can be found in Section \ref{sec.main}. Specific structural constraints such as additive models are discussed in Section \ref{sec.examples}. In Section \ref{sec.wavelets} it is shown that wavelet estimators can only achieve suboptimal rates under the composition assumption. We give an overview of relevant related literature in Section \ref{sec.history}. The proof of the main result together with additional discussion can be found in Section \ref{sec.proofs}.

{\it Notation:} Vectors are denoted by bold letters, e.g. $\bx :=(x_1,\ldots, x_d)^\top.$ As usual, we define $|\bx|_p := (\sum_{i=1}^d |x_i|^p)^{1/p},$ $|\bx |_\infty := \max_i |x_i|,$ $|\bx|_0 := \sum_i \mathbf{1}(x_i \neq 0),$ and write $\|f\|_p := \|f \|_{L^p(D)}$ for the $L^p$-norm on $D,$ whenever there is no ambiguity of the domain $D.$ For two sequences $(a_n)_n$ and $(b_n)_n,$ we write $a_n \lesssim b_n$ if there exists a constant $C$ such that $a_n \leq Cb_n$ for all $n.$ Moreover, $a_n \asymp b_n$ means that $a_n \lesssim b_n$ and $b_n \lesssim a_n.$ We denote by $\log_2$ the logarithm with respect to the basis two and write $\lceil x \rceil$ for the smallest integer $\geq x.$ 

\section{Mathematical definition of multilayer neural networks}
\label{sec.multilayer_NN}

{\quad} \\
{\bf Definitions:} Fitting a multilayer neural network requires the choice of an activation function $\sigma:\mathbb{R}\rightarrow \mathbb{R}$ and the network architecture. Motivated by the importance in deep learning, we study the rectifier linear unit (ReLU) activation function $$\sigma(x) = \max(x,0).$$ For $\bv=(v_1, \ldots, v_r)\in \mathbb{R}^r,$ define the shifted activation function $\sigma_{\bv}: \mathbb{R}^r \rightarrow \mathbb{R}^r$ as 
\begin{align*}
	\sigma_{\bv}
	\left(
	\begin{array}{c}
	y_1 \\
	\vdots \\
	y_r
	\end{array}
	\right)
	=
	\left(
	\begin{array}{c}
	\sigma(y_1-v_1) \\
	\vdots \\
	\sigma(y_r-v_r)
	\end{array}
	\right).
\end{align*}
The network architecture $(L, \bp)$ consists of a positive integer $L$ called the {\it number of hidden layers} or {\it depth} and a {\it width vector} $\bp=(p_0, \ldots, p_{L+1}) \in \mathbb{N}^{L+2}.$ A neural network with network architecture $(L, \bp)$ is then any function of the form 
\begin{align}
	f: \mathbb{R}^{p_0} \rightarrow \mathbb{R}^{p_{L+1}}, \quad \bx \mapsto f(\bx) = W_L  \sigma_{\bv_L}   W_{L-1}  \sigma_{\bv_{L-1}}  \cdots  W_1 \sigma_{\bv_1}  W_0\bx,
	\label{eq.NN}
\end{align}
where $W_i$ is a $p_{i+1} \times p_i$ weight matrix and $\bv_i \in \mathbb{R}^{p_i}$ is a shift vector. Network functions are therefore build by alternating matrix-vector multiplications with the action of the non-linear activation function $\sigma.$ In \eqref{eq.NN}, it is also possible to omit the shift vectors by considering the input $(\bx,1)$ and enlarging the weight matrices by one row and one column with appropriate entries. For our analysis it is, however, more convenient to work with representation \eqref{eq.NN}. To fit networks to data generated from the $d$-variate nonparametric regression model we must have $p_0=d$ and $p_{L+1} =1.$

\begin{figure}
\begin{center}
	\includegraphics[scale=0.3]{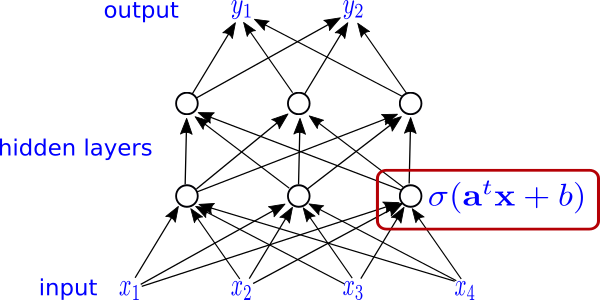} 
\end{center}
\caption{\label{fig.NN} Representation as a direct graph of a network with two hidden layers $L=2$ and width vector $\bp=(4,3,3,2).$}
\end{figure}

In computer science, neural networks are more commonly introduced via their representation as directed acyclic graphs, cf. Figure \ref{fig.NN}. Using this equivalent definition, the nodes in the graph (also called {\it units}) are arranged in layers. The  input layer is the first layer and the output layer the last layer. The layers that lie in between are called hidden layers. The number of hidden layers corresponds to $L$ and the number of units in each layer generates the width vector $\bp.$ Each node/unit in the graph representation stands for a scalar product of the incoming signal with a weight vector which is then shifted and applied to the activation function. 

{\bf Mathematical modeling of deep network characteristics:} Given a network function $f(\bx) = W_L \sigma_{\bv_L}   W_{L-1}  \sigma_{\bv_{L-1}}  \cdots  W_1 \sigma_{\bv_1}  W_0\bx,$ the network parameters are the entries of the matrices $(W_j)_{j=0, \ldots, L}$ and vectors $(\bv_j)_{j=1, \ldots, L}.$ These parameters need to be estimated/learned from the data. 

The aim of this article is to consider a framework that incorporates essential features of modern deep network architectures. In particular, we allow for large depth $L$ and large number of potential network parameters. For the main result, no upper bound on the number of network parameters is needed. Thus we consider high-dimensional settings with more parameters than training data. 

Another characteristic of trained networks is that the size of the learned network parameters is typically not very large. Common network initialization methods initialize the weight matrices $W_j$ by a (nearly) orthogonal random matrix if two successive layers have the same width, cf. \cite{Goodfellow-et-al-2016-Book}, Section 8.4. In practice, the trained network weights are typically not far from the initialized weights. Since in an orthogonal matrix all entries are bounded in absolute value by one, also the trained network weights will not be large.

Existing theoretical results, however, often require that the size of the network parameters tends to infinity. If large parameters are allowed, one can for instance easily approximate step functions by ReLU networks. To be more in line with what is observed in practice, we consider networks with all parameters bounded by one. This constraint can be easily build into the deep learning algorithm by projecting the network parameters in each iteration onto the interval $[-1,1].$

If $\|W_j\|_\infty$ denotes the maximum-entry norm of $W_j,$ the space of network functions with given network architecture and network parameters bounded by one is
\begin{align}
	\mF(L, \bp):= \Big\{f \, \text{of the form \eqref{eq.NN}} \, : \max_{j = 0, \ldots, L} \|W_j\|_{\infty} \vee |\bv_j|_\infty \leq 1 \Big\},
	\label{eq.defi_bd_parameter_space}
\end{align}
with the convention that $\bv_0$ is a vector with all components equal to zero.

In deep learning, sparsity of the neural network is enforced through regularization or specific forms of networks. Dropout for instance sets randomly units to zero and has the effect that each unit will be active only for a small fraction of the data, cf. \cite{srivastava2014}, Section 7.2. In our notation this means that each entry of the vectors $\sigma_{\bv_k} W_{k-1}  \cdots  W_1 \sigma_{\bv_1}  W_0\bx,$ $k=1, \ldots, L$ is zero over a large range of the input space $\bx \in [0,1]^d.$ Convolutional neural networks filter the input over local neighborhoods. Rewritten in the form \eqref{eq.NN} this essentially means that the $W_i$ are banded Toeplitz matrices. All network parameters corresponding to higher off-diagonal entries are thus set to zero. 

In this work we model the network sparsity assuming that there are only few non-zero/active network parameters. If $\|W_j\|_0$ denotes the number of non-zero entries of $W_j$ and $\| |f|_\infty \|_\infty$ stands for the sup-norm of the function $\bx \mapsto |f(\bx)|_\infty,$ then the $s$-sparse networks are given by
\begin{align}
	\begin{split}
	\mF(L, \bp, s)
	&:=\mF(L, \bp, s, F) \\
	&:= \Big\{f \in \mathcal{F}(L, \bp) \, : \sum_{j=0}^L \|W_j\|_0 + |\bv_j|_0 \leq s, \, \big\| |f|_\infty \big\|_\infty \leq F \Big\}.
	\end{split}
	\label{eq.defi_bd_sparse_para_space}
\end{align}
The upper bound on the uniform norm of $f$ is most of the time dispensable and therefore omitted in the notation. We consider cases where the number of network parameters $s$ is small compared to the total number of parameters in the network.

In deep learning, it is common to apply variations of stochastic gradient descent combined with other techniques such as dropout to the loss induced by the log-likelihood (see Section 6.2.1.1 in \cite{Goodfellow-et-al-2016-Book}). For nonparametric regression with normal errors, this coincides with the least-squares loss (in machine learning terms this is the cross-entropy for this model, cf. \cite{Goodfellow-et-al-2016-Book}, p.129). The common objective of all reconstruction methods is to find networks $f$ with small empirical risk $\tfrac 1n \sum_{i=1}^n(Y_i -f (\bX_i))^2.$ For any estimator $\widehat{f}_n$ that returns a network in the class $\mF(L, \bp, s, F)$ we define the corresponding quantity
\begin{align}
	\begin{split}
	&\Delta_n(\widehat f_n, f_0) \\
	&:= E_{f_0}\Big[\frac 1n \sum_{i=1}^n(Y_i -\widehat f_n (\bX_i))^2 - \inf_{f\in\mF(L, \bp, s, F)} \frac 1n \sum_{i=1}^n(Y_i -f(\bX_i))^2 \Big].
	\end{split}
	\label{eq.estimator}
\end{align}
The sequence $\Delta_n(\widehat f_n, f_0)$ measures the difference between the expected empirical risk of $\widehat f_n$ and the global minimum over all networks in the class. The subscript $f_0$ in $E_{f_0}$ indicates that the expectation is taken with respect to a sample generated from the nonparametric regression model with regression function $f_0.$ Notice that $\Delta_n(\widehat f_n, f_0) \geq 0$ and $\Delta_n(\widehat f_n, f_0) =0$ if $\widehat{f}_n$ is an empirical risk minimizer.

To evaluate the statistical performance of an estimator $\widehat f_n,$ we derive bounds for the prediction error
\begin{align*}
	R(\widehat f_n, f_0) := E_{f_0}\big[\big( \widehat f_n(\bX) - f_0(\bX) \big)^2\big],
\end{align*}
with $\bX \stackrel{\mathcal{D}}{=}\bX_1$ being independent of the sample $(\bX_i,Y_i)_i.$

The term $\Delta_n(\widehat f_n, f_0)$ can be related via empirical process theory to constant $\times (R(\widehat f_n, f_0)- R(\widehat f_n^{\ERM}, f_0)) +$ remainder, with $\widehat f_n^{\ERM}$ an empirical risk minimizer. Therefore, $\Delta_n(\widehat f_n, f_0)$ is the key quantity that together with the minimax estimation rate sharply determines the convergence rate of $\widehat f_n$ (up to $\log n$-factors). Determining the decay of $\Delta_n(\widehat f_n, f_0)$ in $n$ for commonly employed methods such as stochastic gradient descent is an interesting problem in its own. We only sketch a possible proof strategy here. Because of the potentially many local minima and saddle points of the loss surface or energy landscape, gradient descent based methods have only a small chance to reach the global minimum without getting stuck in a local minimum first. By making a link to spherical spin glasses, \cite{choromanska2015} provide a heuristic suggesting that the loss of any local minima lies in a band that is lower bounded by the loss of the global minimum. The width of the band depends on the width of the network. If the heuristic argument can be made rigorous, then the width of the band provides an upper bound for $\Delta_n(\widehat f_n, f_0)$ for all methods that converge to a local minimum. This would allow us then to study deep learning without an explicit analysis of the algorithm. For more on the energy landscape, see \cite{liao2017}.


\section{Main results}
\label{sec.main}

The theoretical performance of neural networks depends on the underlying function class. The classical approach in nonparametric statistics is to assume that the regression function is $\beta$-smooth. The minimax estimation rate for the prediction error is then $n^{-2\beta/(2\beta+d)}.$ Since the input dimension $d$ in neural network applications is very large, these rates are extremely slow. The huge sample sizes often encountered in deep learning applications are by far not sufficient to compensate the slow rates. 

We therefore consider a function class that is natural for neural networks and exhibits some low-dimensional structure that leads to input dimension free exponents in the estimation rates. We assume that the regression function $f_0$ is a composition of several functions, that is,
\begin{align}
	f_0 = g_q \circ g_{q-1} \circ \ldots \circ g_1 \circ g_0
  \label{eq.mult_composite_regression}	
\end{align}
with $g_i :[a_i,b_i]^{d_i} \rightarrow [a_{i+1},b_{i+1}]^{d_{i+1}}.$ Denote by $g_i=(g_{ij})_{j=1,\ldots,d_{i+1}}^\top$ the components of $g_i$ and let $t_i$ be the maximal number of variables on which each of the $g_{ij}$ depends on. Thus, each $g_{ij}$ is a $t_i$-variate function. As an example consider the function $f_0(x_1,x_2,x_3) = g_{11}(g_{01}(x_1,x_3), g_{02}(x_1,x_2))$ for which $d_0=3,$ $t_0=2,$ $d_1=t_1=2,$ $d_2=1.$ We always must have $t_i \leq d_i$ and for specific constraints such as additive models, $t_i$ might be much smaller than $d_i.$ The single components $g_0, \ldots, g_q$ and the pairs $(\beta_i, t_i)$ are obviously not identifiable. As we are only interested in estimation of $f_0$ this causes no problems. Among all possible representations, one should always pick one that leads to the fastest estimation rate in Theorem \ref{thm.main} below. 

In the $d$-variate regression model \eqref{eq.mod}, $f_0: [0,1]^d \rightarrow \mathbb{R}$ and thus $d_0=d,$ $a_0=0,$ $b_0=1$ and $d_{q+1}=1.$ One should keep in mind that \eqref{eq.mult_composite_regression} is an assumption on the regression function that can be made independently of whether neural networks are used to fit the data or not. In particular, the number of layers $L$ in the network has not to be the same as $q.$

It is conceivable that for many of the problems for which neural networks perform well a hidden hierarchical input-output relationship of the form \eqref{eq.mult_composite_regression} is present with small values $t_i,$ cf. \cite{Poggio2017, mhaskar2016}. Slightly more specific function spaces, which alternate between summations and composition of functions, have been considered in \cite{horowitz2007, bauer2017}. We provide below an example of a function class that can be decomposed in the form \eqref{eq.mult_composite_regression} but is not contained in these spaces. 

Recall that a function has H\"older smoothness index $\beta$ if all partial derivatives up to order $\lfloor \beta \rfloor$ exist and are bounded and the partial derivatives of order $\lfloor \beta \rfloor$ are $\beta - \lfloor \beta \rfloor$ H\"older, where $\lfloor \beta \rfloor$ denotes the largest integer strictly smaller than $\beta.$ The ball of $\beta$-H\"older functions with radius $K$ is then defined as
\begin{align*}
	\mC_r^\beta(D, K) = \Big\{ 
	&f:D \subset \mathbb{R}^r \rightarrow \mathbb{R} : \\
	&\sum_{\balpha : |\balpha| < \beta}\|\partial^{\balpha} f\|_\infty + \sum_{\balpha : |\balpha |= \lfloor \beta \rfloor } \, \sup_{\stackrel{\bx, \by \in D}{\bx \neq \by}}
	\frac{|\partial^{\balpha} f(\bx) - \partial^{\balpha} f(\by)|}{|\bx-\by|_\infty^{\beta-\lfloor \beta \rfloor}} \leq K
	\Big\},
\end{align*}
where we used multi-index notation, that is, $\partial^{\balpha}= \partial^{\alpha_1}\ldots \partial^{\alpha_r}$ with $\balpha =(\alpha_1, \ldots, \alpha_r)\in \mathbb{N}^r$ and $|\balpha| :=|\balpha|_1.$

We assume that each of the functions $g_{ij}$ has H\"older smoothness $\beta_i.$ Since $g_{ij}$ is also $t_i$-variate, $g_{ij} \in \mC_{t_i}^{\beta_i}([a_i,b_i]^{t_i},K_i)$ and the underlying function space becomes
\begin{align*}
	\mG\big(q, \bd, \bt, \bbeta, K\big)
	:= \Big \{ 
	f= g_q \circ &\ldots \circ g_0 \, : \,  g_i=(g_{ij})_j : [a_i, b_i]^{d_i}\rightarrow [a_{i+1},b_{i+1}]^{d_{i+1}},\\
	 &g_{ij} \in \mC_{t_i}^{\beta_i}\big([a_i,b_i]^{t_i}, K\big), \  \text{for some \ } |a_i|, |b_i| \leq K \Big\},
\end{align*}
with $\bd:=(d_0,\ldots,d_{q+1}),$ $\bt:=(t_0, \ldots, t_q),$ $\bbeta:=(\beta_0, \ldots, \beta_q).$

For estimation rates in the nonparametric regression model, the crucial quantity is the smoothness of $f.$ Imposing smoothness on the functions $g_i,$ we must then find the induced smoothness on $f.$  If, for instance, $q=1,$ $\beta_0, \beta_1\leq 1,$ $d_0=d_1=t_0=t_1=1,$ then $f= g_1 \circ g_0$ and $f$ has smoothness $\beta_0\beta_1,$ cf. \cite{juditsky2009, ray2017}. We should then be able to achieve at least the convergence rate $n^{-2\beta_0\beta_1/(2\beta_0\beta_1+1)}.$ For $\beta_1 >1,$ the rate changes. Below we see that the convergence of the network estimator is described by the effective smoothness indices
\begin{align*}
	\beta_i^* := \beta_i \prod_{\ell=i+1}^q (\beta_{\ell}\wedge 1)
\end{align*}
via the rate
\begin{align}
	\phi_n := \max_{i=0, \ldots, q } n^{-\frac{2\beta_i^*}{2\beta_i^*+ t_i}}.
	\label{eq.phi_def}
\end{align}
Recall the definition of $\Delta_n(\widehat f_n, f_0)$ in \eqref{eq.estimator}. We can now state the main result.

\begin{thm}
\label{thm.main}
Consider the $d$-variate nonparametric regression model \eqref{eq.mod} for composite regression function \eqref{eq.mult_composite_regression} in the class $\mG(q, \bd, \bt, \bbeta, K).$ Let $\widehat f_n$ be an estimator  taking values in the network class $\mF(L, (p_i)_{i=0,\ldots, L+1}, s, F)$ satisfying
\begin{compactitem}
\item[(i)] $F \geq \max(K,1),$
\item[(ii)] $\sum_{i=0}^q \log_2 (4t_i\vee 4\beta_i) \log_2 n \leq L \lesssim n \phi_n,$
\item[(iii)] $n \phi_n \lesssim \min_{i=1, \ldots, L} p_i,$
\item[(iv)] $s \asymp n \phi_n \log n.$
\end{compactitem}
There exist constants $C, C'$ only depending on $q, \bd, \bt, \bbeta, F,$ such that if 
\\ $\Delta_n(\widehat f_n, f_0) \leq C \phi_n L \, \log^2 n,$ then
\begin{align}
	R(\widehat f_n, f_0) \leq C' \phi_n L \, \log^2 n,
	\label{eq.main1}
\end{align}
and if $\Delta_n(\widehat f_n, f_0) \geq C \phi_n L \, \log^2 n,$ then
\begin{align}
	\frac 1{C'} \Delta_n(\widehat f_n, f_0) \leq  R(\widehat f_n, f_0) \leq C' \Delta_n(\widehat f_n, f_0).
	\label{eq.main2}
\end{align}
\end{thm}

In order to minimize the rate $\phi_n  L \, \log^2 n,$ the best choice is to choose $L$ of the order of $\log_2 n.$ The rate in the regime $\Delta_n(\widehat f_n, f_0) \leq C \phi_n \, \log^3 n$ becomes then
\begin{align*}
	R(\widehat f_n, f_0) \leq C' \phi_n \, \log^3 n.
\end{align*}

The convergence rate in Theorem \ref{thm.main} depends on $\phi_n$ and $\Delta_n(\widehat f_n, f_0).$ Below we show that $\phi_n$ is a lower bound for the minimax estimation risk over this class.  Recall that the term $\Delta_n(\widehat f_n, f_0)$ is large if $\widehat f_n$ has a large empirical risk compared to an empirical risk minimizer. Having this term in the convergence rate is unavoidable as it also appears in the lower bound in \eqref{eq.main2}. Since for any empirical risk minimizer the $\Delta_n$-term is zero by definition, we have the following direct consequence of the main theorem. 

\begin{cor}
Let $\widetilde f_n \in \argmin_{f\in \mF(L,\bp,s,F)} \sum_{i=1}^n(Y_i-f(\bX_i))^2$ be an empirical risk minimizer. Under the same conditions as for Theorem \ref{thm.main}, there exists a constant $C',$ only depending on $q, \bd, \bt, \bbeta, F,$ such that
\begin{align}
	R(\widetilde f_n, f_0) \leq C' \phi_n L \, \log^2 n.
\end{align}
\end{cor}

Condition $(i)$ in Theorem \ref{thm.main} is very mild and only states that the network functions should have at least the same supremum norm as the regression function. From the other assumptions in Theorem \ref{thm.main} it becomes clear that there is a lot of flexibility in picking a good network architecture as long as the number of active parameters $s$ is taken to be of the right order. Interestingly, to choose a network depth $L,$ it is sufficient to have an upper bound on the $t_i \leq d_i$ and the smoothness indices $\beta_i.$ The network width can be chosen independent of the smoothness indices by taking for instance $n \lesssim \min_i p_i.$ One might wonder whether for an empirical risk minimizer the sparsity $s$ can be made adaptive by minimizing a penalized least squares problem with sparsity inducing penalty on the network weights. It is conceivable that a complexity penalty of the form $\lambda s$ will lead to adaptation if the regularization parameter $\lambda$ is chosen of the correct order. From a practical point of view, it is more interesting to study $\ell^1/\ell^2$-weight decay. As this requires much more machinery, the question will be moved to future work. 

The number of network parameters in a fully connected network is of the order $\sum_{i=0}^L p_i p_{i+1}.$ This shows that Theorem \ref{thm.main} requires sparse networks. More specifically, the network has at least $\sum_{i=1}^{L} p_i- s$ completely inactive nodes, meaning that all incoming signal is zero. The choice $s \asymp n \phi_n \log n$ in condition $(iv)$ balances the squared bias and the variance. From the proof of the theorem convergence rates can also be derived if $s$ is chosen of a different order.

For convenience, Theorem \ref{thm.main} is stated without explicit constants. The proofs, however, are non-asymptotic although we did not make an attempt to minimize the constants. It is well-known that deep learning outperforms other methods only for large sample sizes. This indicates that the method might be able to adapt to underlying structure in the signal and therefore achieving fast convergence rates but with large constants or remainder terms which spoil the results for small samples. 

The proof of the risk bounds in Theorem \ref{thm.main} is based on the following oracle-type inequality.

\begin{thm}
\label{thm.oracle_ineq}
Consider the $d$-variate nonparametric regression model \eqref{eq.mod} with unknown regression function $f_0,$ satisfying $\|f_0 \|_\infty \leq F$ for some $F\geq 1.$ Let $\widehat f_n$ be any estimator taking values in the class $\mF(L, \bp, s, F)$ and let $\Delta_n(\widehat f_n, f_0)$ be the quantity defined in \eqref{eq.estimator}. For any $\eps\in (0,1],$ there exists a constant $C_\eps,$ only depending on $\eps,$ such that with 
\begin{align*}
	\tau_{\eps,n} :=  C_\eps F^2\frac{(s+1) \log (n(s+1)^{L}p_0p_{L+1})}{n},
\end{align*}
\begin{align*}
	&(1-\eps)^2 \Delta_n(\widehat f_n, f_0) - \tau_{\eps,n} \leq R(\widehat f_n, f_0)  \\
	&\leq (1+\eps)^2 \Big(\inf_{f\in \mF(L, \bp, s, F)} \big\| f - f_0\big\|_\infty^2 +\Delta_n(\widehat f_n, f_0) \Big) + \tau_{\eps,n}.
\end{align*}
\end{thm}

One consequence of the oracle inequality is that the upper bounds on the risk become worse if the number of layers increases. In practice it also has been observed that too many layers lead to a degradation of the performance, cf. \cite{he2016}, \cite{he2015}, Section 4.4 and \cite{srivastava2015}, Section 4. Residual networks can overcome this problem. But they are not of the form \eqref{eq.NN} and will need to be analyzed separately. 

One may wonder whether there is anything special about ReLU networks compared to other activation functions. A close inspection of the proof shows that two specific properties of the ReLU function are used. 

One of the advantages of deep ReLU networks is the projection property 
\begin{align}
	\sigma\circ \sigma =\sigma
	\label{eq.projection}
\end{align}
that we can use to pass a signal without change through several layers in the network. This is important since the approximation theory is based on the construction of smaller networks for simpler tasks that might not all have the same network depth. To combine these subnetworks we need to synchronize the network depths by adding hidden layers that do not change the output. This can be easily realized by choosing the weight matrices in the network to be the identity (assuming equal network width in successive layers) and using \eqref{eq.projection}, see also \eqref{eq.add_layers}. This property is not only a theoretical tool. To pass an outcome without change to a deeper layer is also often helpful in practice and realized by so called skip connections in which case they do not need to be learned from the data. A specific instance are residual networks with ReLU activation function \cite{he2016} that are successfully applied in practice. The difference to standard feedforward networks is that if all networks parameters are set to zero in a residual network, the network becomes essentially the identity map. For other activation functions it is much harder to approximate the identity.

Another advantage of the ReLU activation is that all network parameters can be taken to be bounded in absolute value by one. If all network parameters are initialized by a value in $[-1,1],$ this means that each network parameter only need to be varied by at most two during training. It is unclear whether other results in the literature for non-ReLU activation functions hold for bounded network parameters. An important step is the approximation of the square function $x\mapsto x^2.$ For any twice differentiable and non-linear activation function, the classical approach to approximate the square function by a network is to use rescaled second order differences $(\sigma(t+2xh)-2\sigma(t+xh)+\sigma(xh))/(h^2\sigma''(t)) \rightarrow x^2$ for $h \rightarrow 0$ and a $t$ with $\sigma''(t) \neq 0$. To achieve a sufficiently good approximation, we need to let $h$ tend to zero with the sample size, making some of the network parameters necessarily very large. 

The $\log^2 n$-factor in the convergence rate $\phi_n L \log^2 n$ is likely an artifact of the proof. Next we show that $\phi_n$ is a lower bound for the minimax estimation risk over the class $\mG(q, \bd, \bt, \bbeta, K)$ in the interesting regime $t_i \leq \min(d_0, \ldots, d_{i-1})$ for all $i.$ This means that no dimensions are added on deeper abstraction levels in the composition of functions. In particular, it avoids that $t_i$ is larger than the input dimension $d_0.$ Outside of this regime, it is hard to determine the minimax rate and in some cases it is even possible to find another representation of $f$ as a composition of functions which yields a faster convergence rate. 

\begin{thm}
\label{thm.lb}
Consider the nonparametric regression model \eqref{eq.mod} with $\bX_i$ drawn from a distribution with Lebesgue density on $[0,1]^d$ which is lower and upper bounded by positive constants. For any non-negative integer $q,$ any dimension vectors $\bd$ and $\bt$ satisfying $t_i \leq \min(d_0, \ldots, d_{i-1})$ for all $i,$ any smoothness vector $\bbeta$ and all sufficiently large constants $K>0,$ there exists a positive constant $c,$ such that 
\begin{align*}
	\inf_{\widehat f_n} \, \sup_{f_0 \in \mG(q, \bd, \bt, \bbeta, K)} R\big( \widehat f_n, f_0\big) \geq c  \phi_n,
\end{align*}
where the $\inf$ is taken over all estimators $\widehat f_n.$
\end{thm}

The proof is deferred to Section \ref{sec.proofs}. To illustrate the main ideas, we provide a sketch here. For simplicity assume that $t_i =d_i =1$ for all $i$. In this case, the functions $g_i$ are univariate and real-valued. Define $i^* \in \argmin_{i=0, \ldots, q } \beta_i^*/(2\beta_i^*+ 1)$ as an index for which the estimation rate is obtained. For any $\alpha>0,$  $x^\alpha$ has H\"older smoothness $\alpha$ and for $\alpha=1,$ the function is infinitely often differentiable and has finite H\"older norm for all smoothness indices. Set $g_\ell(x) = x$ for $\ell< i^*$ and $g_\ell(x) = x^{\beta_\ell \wedge 1}$ for $\ell >i^*.$ Then,
\begin{align*}
	f_0(x) = g_q \circ g_{q-1} \circ \ldots \circ g_1 \circ g_0(x) = \big(g_{i^*}(x)\big)^{\prod_{\ell=i^*+1}^q \beta_\ell \wedge 1}.
\end{align*}
Assuming for the moment uniform random design, the Kullback-Leibler divergence is $\KL(P_f, P_g) = \tfrac n 2 \| f-g\|_2^2.$ Take a kernel function $K$ and consider $\widetilde g(x) = h^{\beta_{i^*}} K(x/h).$ Under standard assumptions on $K,$ $\widetilde g$ has H\"older smoothness index $\beta_{i^*}.$ Now we can generate two hypotheses $f_{00}(x)=0$ and $f_{01}(x) = (h^{\beta_{i^*}} K(x/h))^{\prod_{\ell=i^*+1}^q \beta_\ell \wedge 1}$ by taking $g_{i^*}(x) =0$ and $g_{i^*}(x) = \widetilde g(x).$ Therefore, $|f_{00}(0) -f_{01}(0)| \gtrsim h^{\beta_{i^*}^*}$ assuming that $K(0)>0.$ For the Kullback-Leibler divergence, we find  $\KL(P_{f_{00}}, P_{f_{01}}) \lesssim nh^{2\beta_{i^*}^*+1}.$ Using Theorem 2.2 (iii) in \cite{tsybakov2009}, this shows that the pointwise rate of convergence is $n^{-2\beta_{i^*}^*/(2\beta_{i^*}^*+ 1)}= \max_{i=0, \ldots, q} n^{-2\beta_i^*/(2\beta_i^*+ 1)}.$ This matches with the upper bound since $t_i=1$ for all $i.$ For lower bounds on the prediction error, we generalize the argument to a multiple testing problem.

The $L^2$-minimax rate coincides in most regimes with the sup-norm rate obtained in Section 4.1 of \cite{juditsky2009} for composition of two functions. But unlike the classical nonparametric regression model, the minimax estimation rates for $L^2$-loss and sup-norm loss differ for some setups by a polynomial power in the sample size $n.$

There are several recent results in approximation theory that provide lower bounds on the number of required network weights $s$ such that all functions in a function class can be approximated by a $s$-sparse network up to some prescribed error, cf. for instance \cite{boelcskei2017}. Results of this flavor can also be quite easily derived by combining the minimax lower bound with the oracle inequality. The argument is that if the same approximation rates would hold for networks with less parameters then we would obtain rates that are faster than the minimax rates, which clearly is a contradiction. This provides a new statistical route to establish approximation theoretic properties. 

\begin{lem}
\label{lem.approx}
Given $\beta, K>0,$ $d\in \mathbb{N},$ there exist constants $c_1, c_2$ only depending on $\beta, K, d,$ such that if $$s\leq c_1\frac{\eps^{-d/\beta}}{L\log(1/\eps)}$$ for some $\eps \leq c_2,$ then for any width vector $\bp$ with $p_0=d$ and $p_{L+1}=1,$
\begin{align*}
	\sup_{f_0 \in \mC_d^\beta([0,1]^d, K)}\, \inf_{f\in \mF(L,\bp, s)} \| f - f_0 \|_\infty \geq \eps.
\end{align*}
\end{lem}

A more refined argument using Lemma \ref{lem.oracle_gen} instead of Theorem \ref{thm.oracle_ineq} yields also lower bounds for $L^2.$

\section{Examples of specific structural constraints}
\label{sec.examples}

In this section we discuss several well-studied special cases of compositional constraints on the regression function.

{\it Additive models:} In an additive model the regression function has the form
\begin{align*}
	f_0(x_1, \ldots, x_d) = \sum_{i=1}^d f_i(x_i).
\end{align*}
This  can be written as a composition of functions 
\begin{align}
	f_0=g_1 \circ g_0
	\label{eq.composite_regression}
\end{align}
with $g_0(\bx)=(f_1(x_1), \ldots,f_d(x_d))^\top$ and $g_1(\by)= \sum_{j=1}^d y_j.$ Consequently, $g_0: [0,1]^d \rightarrow \mathbb{R}^d$ and $g_1: \mathbb{R}^d \rightarrow \mathbb{R}$ and thus $d_0=d,$ $t_0=1,$ $d_1=t_1=d,$ $d_2=1.$ Equation \eqref{eq.composite_regression} decomposes the original function into one function where each component only depends on one variable only and another function that depends on all variables but is infinitely smooth. For both types of functions fast rates can be obtained that do not suffer from the curse of dimensionality. This explains then the fast rate that can be obtained for additive models.

Suppose that $f_i \in \mC_1^\beta([0,1],K)$ for $i=1,\ldots, d.$ Then, $f: \, [0,1]^d \stackrel{g_0}{\longrightarrow} [-K,K]^d \stackrel{g_1}{\longrightarrow} [-Kd, Kd].$ Since for any $\gamma>1,$ $g_1 \in \mC_d^\gamma ([-K,K]^d,(K+1)d),$ 
\begin{align*}
	f_0 \in \mG\big(1, (d,d,1), (1,d), (\beta, (\beta\vee 2) d), (K+1) d\big).
\end{align*}
For network architectures $\mF(L, \bp, s, F)$ satisfying $F\geq (K+1)d,$ $2\log_2 (4(\beta\vee 2) d)\log n\leq L\lesssim \log n,$ $n^{1/(2\beta+1)}\lesssim \min_i p_i$ and $s \asymp n^{1/(2\beta+1)}\log n,$ we thus obtain by Theorem \ref{thm.main},
\begin{align*}
		R(\widehat f_n, f_0) \lesssim n^{-\frac{2\beta}{2\beta+1}} \log^3n + \Delta( \widehat f_n, f_0).
\end{align*}
This coincides up to the $\log^3 n$-factor with the minimax estimation rate.


{\it Generalized additive models:} Suppose the regression function is of the form
\begin{align*}
	f_0(x_1, \ldots, x_d) = h\Big(\sum_{i=1}^d f_i(x_i) \Big),
\end{align*}
for some unknown link function $h: \mathbb{R}\rightarrow \mathbb{R}.$ This can be written as composition of three functions $f_0=g_2 \circ g_1 \circ g_0$ with $g_0$ and $g_1$ as before and $g_2 = h.$ If $f_i \in \mC_1^\beta([0,1],K)$ and $h \in \mC_1^\gamma(\mathbb{R},K),$ then $f_0: \, [0,1]^d \stackrel{g_0}{\longrightarrow} [-K,K]^d \stackrel{g_1}{\longrightarrow} [-Kd, Kd] \stackrel{g_2}{\longrightarrow} [-K,K].$ Arguing as for additive models, 
\begin{align*}
	f_0\in \mG \Big( 2, (d,d,1,1), (1,d,1), (\beta, (\beta \vee 2) d, \gamma), (K+1) d\Big).
\end{align*}
For network architectures satisfying the assumptions of Theorem \ref{thm.main}, the bound on the estimation rate becomes
\begin{align}
		R(\widehat f_n, f_0) \lesssim \Big( n^{-\frac{2\beta (\gamma \wedge 1)}{2\beta (\gamma \wedge 1)+1}} + n^{-\frac{2\gamma}{2\gamma+1}}\Big) \log^3n + \Delta( \widehat f_n, f_0).
		\label{eq.gen_add_mod_rate}
\end{align}
Theorem \ref{thm.lb} shows that $n^{-2\beta (\gamma \wedge 1)/(2\beta (\gamma \wedge 1)+1)} + n^{-2\gamma/(2\gamma+1)}$ is also a lower bound. Let us also remark that for the special case $\beta=\gamma \geq 2$ and $\beta, \gamma$ integers, Theorem 2.1 of \cite{horowitz2007} establishes the estimation rate $n^{-2\beta/(2\beta+1)}.$

{\it Sparse tensor decomposition:} Assume that the regression function $f_0$ has the form
\begin{align}
	f_0(\bx) = \sum_{\ell=1}^N a_\ell \prod_{i=1}^d f_{i\ell}(x_i),
	\label{eq.mult_structure}
\end{align}
for fixed $N,$ real coefficients $a_\ell$ and univariate functions $f_{i\ell}.$ Especially, if $N=1,$ this is the same as imposing a product structure on the regression function $f_0(\bx) = \prod_{i=1}^d f_i(x_i).$ The function class spanned by such sparse tensor decomposition can be best explained by making a link to series estimators. Series estimators are based on the idea that the unknown function is close to a linear combination of few basis functions, where the approximation error depends on the smoothness of the signal. This means that any $L^2$-function can be approximated by $f_0(\bx) \approx \sum_{\ell=1}^N a_\ell \prod_{i=1}^d \phi_{i\ell}(x_i)$ for suitable coefficients $a_\ell$ and functions $\phi_{i\ell}.$

Whereas series estimators require the choice of a basis, for neural networks to achieve fast rates it is enough that \eqref{eq.mult_structure} holds. The functions $f_{i\ell}$ can be unknown and do not need to be orthogonal.

We can rewrite \eqref{eq.mult_structure} as a composition of functions $f_0=g_2 \circ g_1 \circ g_0$ with $g_0(\bx)= (f_{i\ell}(x_i))_{i,\ell},$ $g_1=(g_{1j})_{j=1, \ldots,N}$ performing the $N$ multiplications $\prod_{i=1}^d $ and $g_2(\by)= \sum_{\ell=1}^N a_\ell y_\ell.$ Observe that $t_0=1$ and $t_1=d.$ Assume that $f_{i\ell} \in \mC_1^\beta([0,1], K)$ for $K\geq 1$ and $\max_\ell |a_\ell | \leq 1.$ Because of $g_{1,j} \in \mC_d^\gamma([-K,K]^d, 2^dK^d)$ for all $\gamma \geq d+1$ and $g_2 \in \mC_N^{\gamma'}([-2^dK^d,2^dK^d]^N, N(2^dK^d+1))$ for $\gamma' > 1,$ we have $[0,1]^d \stackrel{g_0}{\longrightarrow} [-K,K]^{Nd} \stackrel{g_1}{\longrightarrow} [-2^dK^d,2^dK^d]^{N} \stackrel{g_2}{\longrightarrow} [-N(2^dK^d+1),N(2^dK^d+1)]$ and  
\begin{align*}
	f_0\in \mG\Big( 2, (d,Nd,N,1), (1, d, Nd), (\beta, \beta d \vee (d+1), N \beta +1), N(2^dK^d+1)\Big).
\end{align*}
For networks with architectures satisfying $3\log_2 (4(\beta+1)(d+1)N)\log_2 n \leq L \lesssim \log n,$ $n^{1/(2\beta+1)} \lesssim \min_i p_i$ and $s\asymp n^{1/(2\beta+1)} \log n,$ Theorem \ref{thm.main} yields the rate
\begin{align*}
	R(\widehat f_n , f_0) \lesssim n^{-\frac{2\beta}{2\beta+1}} \log^3 n + \Delta( \widehat f_n, f_0),
\end{align*}
and the exponent in the rate does not depend on the input dimension $d.$

\section{Suboptimality of wavelet series estimators} 
\label{sec.wavelets}

As argued before the composition assumption in \eqref{eq.mult_composite_regression} is very natural
 and generalizes many structural constraints such as additive models. In this section, we show that wavelet series estimators are unable to take advantage from the underlying composition structure in the regression function and achieve in some setups much slower convergence rates. 

More specifically, we consider general additive models of the form $f_0(\bx )=h(x_1 + \ldots + x_d)$ with $h\in \mC_1^\alpha([0,d],K).$ This can also be viewed as a special instance of the single index model, where the aim is not to estimate $h$ but $f_0.$ Using \eqref{eq.gen_add_mod_rate}, the prediction error of neural network reconstructions with small empirical risk and depth $L \asymp \log n$ is then bounded by $n^{-2\alpha/(2\alpha+1)}\log^3n.$ The lower bound below shows that wavelet series estimators cannot converge faster than with the rate $n^{-2\alpha/(2\alpha+d)}.$ This rate can be much slower if $d$ is large. Wavelet series estimators thus suffer in this case from the curse of dimensionality while neural networks achieve fast rates. 

Consider a compact wavelet basis of $L^2(\mathbb{R})$ restricted to $L^2[0,1],$ say $(\psi_{\blambda}, \blambda \in \bLambda),$ cf. \cite{cohen1993}. Here, $\bLambda=\{(j,k) : j=-1,0,1,\ldots; k\in I_j\}$ with $k$ ranging over the index set $I_j$ and $\psi_{-1,k}:= \phi(\cdot - k)$ are the shifted scaling functions. Then, for any function $f\in L^2[0,1]^d,$
\begin{align*}
	f(\bx) = \sum_{ (\blambda_1, \ldots, \blambda_d) \in \bLambda\times \ldots \times \bLambda} d_{\blambda_1 \ldots \blambda_d}(f) \prod_{r=1}^d \psi_{\blambda_r}(x_r),
\end{align*}
with convergence in $L^2[0,1]$ and
\begin{align*}
	 d_{\blambda_1 \ldots \blambda_d}(f) := \int f(\bx) \prod_{r=1}^d \psi_{\blambda_r}(x_r) \, d\bx
\end{align*}
the wavelet coefficients.

To construct a counterexample, it is enough to consider the nonparametric regression model $Y_i = f_0(\bX_i)+\eps_i,$ $i=1, \ldots,n$ with uniform design $\bX_i:=(U_{i,1},\ldots,  U_{i,d}) \sim \Unif[0,1]^d.$ The empirical wavelet coefficients are
\begin{align*}
	\widehat d_{\blambda_1 \ldots \blambda_d}(f_0) = \frac 1n \sum_{i=1}^n Y_i \prod_{r=1}^d \psi_{\blambda_r}(U_{i,r}).
\end{align*}
Because of $E[\widehat d_{\blambda_1 \ldots \blambda_d}(f_0)] = d_{\blambda_1 \ldots \blambda_d}(f_0),$ this gives unbiased estimators for the wavelet coefficients. By the law of total variance, 
\begin{align*}
	\Var\big(\widehat d_{\blambda_1 \ldots \blambda_d}(f_0)\big) 
	&= \frac 1n \Var\Big( Y_1 \prod_{r=1}^d \psi_{\blambda_r}(U_{1,r})\Big) \\
	&\geq \frac 1n E\Big[\Var\Big( Y_1 \prod_{r=1}^d \psi_{\blambda_r}(U_{1,r}) \, \Big | U_{1,1,}, \ldots, U_{1,d}\Big)\Big] \\
	&= \frac 1n.
\end{align*}
For the lower bounds we may assume that the smoothness indices are known. For estimation, we can truncate the series expansion on a resolution level that balances squared bias and variance of the total estimator. More generally, we study estimators of the form
\begin{align}
	\widehat f_n(\bx) = \sum_{ (\blambda_1, \ldots, \blambda_d) \in I} \widehat d_{\blambda_1 \ldots \blambda_d}(f_0) \prod_{r=1}^d \psi_{\blambda_r}(x_r),
	\label{eq.def_wavelet_estimator}
\end{align}
for an arbitrary subset $I \subset \Lambda \times \ldots \times \Lambda$. Using that the design is uniform,
\begin{align}
	R(\widehat f_n, f_0) 
	&= \sum_{(\blambda_1, \ldots, \blambda_d) \in I} E\big[\big(\widehat d_{\blambda_1 \ldots \blambda_d}(f_0) -  d_{\blambda_1 \ldots \blambda_d}(f_0)\big)^2 \big] +  \sum_{(\blambda_1, \ldots, \blambda_d) \in I^c} d_{\blambda_1 \ldots \blambda_d}(f_0)^2 \notag \\
	&\geq 
	\sum_{(\blambda_1, \ldots, \blambda_d) \in \Lambda \times \ldots \times \Lambda} \frac{1}{n} \wedge d_{\blambda_1 \ldots \blambda_d}(f_0)^2.
	\label{eq.wav_est_risk_lb}
\end{align}
By construction, $\psi \in L^2(\mathbb{R})$ has compact support, We can therefore without loss of generality assume that $\psi$ is zero outside of $[0,2^q]$ for some integer $q>0.$ 

\begin{lem}
\label{lem.wav_decay}
Let $q$ be as above and set $\nu := \lceil \log_2 d \rceil +1.$ For any $0< \alpha \leq 1$ and any $K>0,$ there exists a non-zero constant $c(\psi, d)$ only depending on $d$ and properties of the wavelet function $\psi$ such that for any $j,$ we can find a function $f_{j,\alpha}(\bx)=h_{j,\alpha}(x_1+\ldots + x_d)$ with $h_{j,\alpha} \in \mC_1^\alpha([0,d], K)$ satisfying
\begin{align*}
		d_{(j,2^{q+\nu} p_1)\ldots (j,2^{q+\nu} p_d)}(f_{j,\alpha}) = c(\psi, d) K 2^{-\frac j 2(2\alpha+ d)}
\end{align*}
for all $p_1, \ldots, p_d\in \{0,1, \ldots, 2^{j-q-\nu}-1\}.$
\end{lem}

\begin{thm}
\label{thm.wavelet_lb}
If $\widehat f_n$ denotes the wavelet estimator \eqref{eq.def_wavelet_estimator} for a compactly supported wavelet $\psi$ and an arbitrary index set $I,$ then, for any $0< \alpha \leq 1$ and any H\"older radius $K>0,$
\begin{align*}
	\sup_{f_0(\bx)=h(\sum_{r=1}^d x_r), \, h \in \mC_1^\alpha([0,d],K)} R(\widehat f_n , f_0) \gtrsim n^{-\frac{2\alpha}{2\alpha+d}}.
\end{align*} 
\end{thm}

A close inspection of the proof shows that the theorem even holds for $0< \alpha \leq r$ with $r$ the smallest positive integer for which $\int x^r \psi(x) dx \neq 0.$

\section{A brief summary of related statistical theory for neural networks}
\label{sec.history}

This section is intended as a condensed overview on related literature summarizing main proving strategies for bounds on the statistical risk. An extended summary of the work until the late 90's is given in \cite{pinkus1999}. To control the stochastic error of neural networks, bounds on the covering entropy and VC dimension can be found in the monograph \cite{anthony1999}. A challenging part in the analysis of neural networks is the approximation theory for multivariate functions. We first recall results for shallow neural networks, that is, neural networks with one hidden layer.

{\bf Shallow neural networks:} A shallow network with one output unit and width vector $(d,m,1)$ can be written as
\begin{align}
	f_m(\bx)=\sum_{j=1}^{m}c_j\sigma\big(\mathbf{w}_j^\top\bx+v_j\big), \quad \bw_j \in \mathbb{R}^d, \ v_j,c_j \in \mathbb{R}.
	\label{eq.shallow_nw}
\end{align}
The universal approximation theorem states that a neural network with one hidden layer can approximate any continuous function $f$ arbitrarily well with respect to the uniform norm provided there are enough hidden units, cf. \cite{Hornik1989, Hornik1990,Cybenko1989, Leshno1993, stinchcombe1999}. If $f$ has a derivative $f'$, then the derivative of the neural network also approximates $f'.$ The number of required hidden units might be, however, extremely large, cf. \cite{pascanu2013} and \cite{montufar2013}. There are several proofs for the universal approximation theorem based on the Fourier transform, the Radon transform and the Hahn-Banach theorem \cite{SCARSELLI199815}. 

The proofs can be sharpened in order to obtain rates of convergence. In \cite{mccaffrey1994} the convergence rate $n^{-2\beta/(2\beta+d+5)}$ is derived. Compared with the minimax estimation rate this is suboptimal by a polynomial factor. The reason for the loss of performance with this approach is that rewriting the function as a network requires too many parameters. 

In \cite{barron1993, barron1994, klusowski2016a, klusowski2016b} a similar strategy is used to derive the rate $C_f(d\tfrac{\log n}n)^{1/2}$ for the squared $L^2$-risk, where $C_f:=\int |\omega|_1|\mathcal{F}f (\omega)| d\omega$ and $\mathcal{F}f$ denotes the Fourier transform of $f.$ If $C_f< \infty$ and $d$ is fixed the rate is always $n^{-1/2}$ up to logarithmic factors. Since $\sum_i \|\partial_i f\|_\infty \leq C_f,$ this means that $C_f< \infty$ can only hold if $f$ has H\"older smoothness at least one. This rate is difficult to compare with the standard nonparametric rates except for the special case $d=1,$ where the rate is suboptimal compared with the minimax rate $n^{-2/(2+d)}$ for $d$-variate functions with smoothness one. More interestingly, the rate $C_f(d\tfrac{\log n}n)^{1/2}$ shows that neural networks can converge fast if the underlying function satisfies some additional structural constraint. The same rate can also be obtained by a Fourier series estimator, see \cite{candes2002}, Section 1.7.  In a similar fashion, \cite{Bach2017} studies abstract function spaces on which shallow networks achieve fast convergence rates.

{\bf Results for multilayer neural networks:} In \cite{mhaskar1993} it is shown how to approximate a polytope by a neural network with two hidden layers. Based on this result, \cite{kohler2005} uses two-layer neural networks with sigmoidal activation function and achieves the nonparametric rate $n^{-2\beta/(2\beta+d)}$ up to $\log n$-factors for $\beta \leq 1.$ This is extended in \cite{kohler2017} to a composition assumption and further generalized to $\beta >1$ in the recent article \cite{bauer2017}. Unfortunately, the result requires that the activation function is at least as smooth as the signal, cf. Theorem 1 in \cite{bauer2017} and therefore rules out the ReLU activation function. 

The activation function $\sigma(x)=x^2$ is not of practical relevance but has some interesting theory. Indeed, with one hidden layer, we can generate quadratic polynomials and with $L$ hidden layers polynomials of degree $2^L.$ For this activation function, the role of the network depth is the polynomial degree and we can use standard results to approximate functions in common function classes. A natural generalization is the class of activation functions satisfying $\lim_{x\rightarrow -\infty} x^{-k}\sigma(x) =0$ and $\lim_{x\rightarrow +\infty} x^{-k}\sigma(x) =1.$ 

If the growth is at least quadratic ($k\geq 2$), the approximation theory has been derived in \cite{mhaskar1993} for deep networks with number of layers scaling with $\log d.$ The same class has also been considered recently in \cite{boelcskei2017}. For the approximations to work, the assumption $k\geq 2$ is crucial and the same approach does not generalize to the ReLU activation function, which satisfies the growth condition with $k=1$, and always produces functions that are piecewise linear in the input. 

Approximation theory for the ReLU activation function has been only recently developed in \cite{telgarsky2016, liang2016, yarotski2017, suzuki2017}. The key observation is that there are specific deep networks with few units which approximate the square function well. In particular, the function approximation presented in \cite{yarotski2017} is essential for our approach and we use a similar strategy to construct networks that are close to a given function. We are, however, interested in a somehow different question. Instead of deriving existence of a network architecture with good approximation properties, we show that for any network architecture satisfying the conditions of Theorem \ref{thm.main} good approximation rates are obtainable. An additional difficulty in our approach is the boundedness of the network parameters. 

\section{Proofs}
\label{sec.proofs}

\subsection{Embedding properties of network function classes}
\label{sec.embedding_props_NNs}

For the approximation of a function by a network, we first construct smaller networks computing simpler objects. Let $\bp = (p_0, \ldots,  p_{L+1})$ and $\bp' = (p_0',  \ldots, p_{L+1}').$ To combine networks, we make frequently use of the following rules. 

{\itshape Enlarging:} $\mF(L, \bp, s) \subseteq \mF(L, \bq, s')$ whenever $\bp \leq \bq$ componentwise and $s\leq s'.$

{\itshape Composition:} Suppose that $f \in \mF(L, \bp)$ and  $g \in \mF(L',\bp')$ with $p_{L+1} =p_0'.$ For a vector $\bv \in \mathbb{R}^{p_{L+1}}$ we define the composed network $g \circ \sigma_{\bv}(f) $ which is in the space $\mF(L+L'+1, (\bp, p_1', \ldots, p_{L'+1}')).$ In most of the cases that we consider, the output of the first network is non-negative and the shift vector $\bv$ will be taken to be zero. 

{\itshape Additional layers/depth synchronization:} To synchronize the number of hidden layers for two networks, we can add additional layers with identity weight matrix, such that
\begin{align}
	\mF(L, \bp,s) \subset \mF(L+q, (\underbrace{p_0,\ldots,p_0}_{q\text{ times}} , \bp), s+qp_0).
	\label{eq.add_layers}
\end{align}

{\itshape Parallelization:} Suppose that $f,g$ are two networks with the same number of hidden layers and the same input dimension, that is, $f \in \mF(L, \bp)$ and $g \in \mF(L, \bp')$ with $p_0=p_0'.$ The parallelized network $(f,g)$ computes $f$ and $g$ simultaneously in a joint network in the class $\mF(L, (p_0, p_1+p_1', \ldots, p_{L+1}+p_{L+1}')).$ 

{\itshape Removal of inactive nodes:} We have 
\begin{align}
	\mF(L, \bp, s)=\mF\big(L, (p_0, p_1\wedge s, p_2\wedge s,  \ldots, p_L\wedge s, p_{L+1}), s\big).
	\label{eq.removal_nodes_identity}	
\end{align}
To see this, let $f(\bx)= W_L\sigma_{\bv_L} W_{L-1} \ldots \sigma_{\bv_1} W_0 \bx \in \mF(L, \bp, s).$ If all entries of the $j$-th column of $W_i$ are zero, then we can remove this column together with the $j$-th row in $W_{i-1}$ and the $j$-th entry of $\bv_i$ without changing the function. This shows then that $f \in \mF(L, (p_0,\ldots, p_{i-1}, p_i-1,p_{i+1}, \ldots, p_{L+1}), s).$ Because there are $s$ active parameters, we can iterate this procedure at least $p_i-s$ times for any $i=1, \ldots, L.$  This proves $f \in \mF\big(L, (p_0, p_1\wedge s, p_2\wedge s,  \ldots, p_L\wedge s, p_{L+1}), s\big).$

We frequently make use of the fact that for a fully connected network in $\mF(L, \bp),$ there are $\sum_{\ell=0}^L p_\ell p_{\ell+1}$ weight matrix parameters and $\sum_{\ell=1}^{L} p_\ell$ network parameters coming from the shift vectors. The total number of parameters is thus
\begin{align}
	\sum_{\ell=0}^L (p_\ell +1) p_{\ell+1} -p_{L+1}.
	\label{eq.nr_of_param_in_net}
\end{align}

\begin{thm}
\label{thm.approx_network_One_fct}
For any function $f\in \mC_r^\beta([0,1]^r, K)$ and any integers $m \geq 1$ and $N \geq  (\beta+1)^r \vee (K+1)e^r,$ there exists a network $$\widetilde f \in \mF\big(L, \big(r, 6(r+\lceil \beta\rceil)N, \ldots, 6(r+\lceil \beta\rceil)N,1\big), s, \infty\big)$$ with depth $$L=8+(m+5)(1+\lceil \log_2 (r\vee \beta) \rceil)$$ and number of parameters
\begin{align*}
	s\leq 141 (r+\beta+1)^{3+r} N (m+6),
\end{align*}
such that 
\begin{align*}
	\| \widetilde f - f\|_{L^\infty([0,1]^r)}\leq  (2K+1)(1+r^2+\beta^2) 6^r N2^{-m}+ K3^\beta N^{-\frac{\beta}r}.
\end{align*}
\end{thm}

The proof of the theorem is given in the supplement. The idea is to first build networks that for given input $(x,y)$ approximately compute the product $xy.$ We then split the input space into small hyper-cubes and construct a network that approximates a local Taylor expansion on each of these hyper-cubes.

Based on Theorem \ref{thm.approx_network_One_fct}, we can now construct a network that approximates $f=g_q \circ \ldots \circ g_0.$ In a first step, we show that $f$ can always be written as composition of functions defined on hypercubes $[0,1]^{t_i}.$ As in the previous theorem, let $g_{ij} \in \mC_{t_i}^{\beta_i}([a_i,b_i]^{t_i},K_i)$ and assume that $K_i \geq 1.$ For $i=1, \ldots, q-1,$ define
\begin{align*}
	h_0 := \frac{g_0}{2K_0} + \frac 12, \quad h_i := \frac{g_i(2K_{i-1}\cdot - K_{i-1})}{2K_i} + \frac 12, \quad h_q= g_q(2K_{q-1}\cdot -K_{q-1}).
\end{align*}
Here, $2K_{i-1}\bx - K_{i-1}$ means that we transform the entries by $2K_{i-1}x_j - K_{i-1}$ for all $j.$ Clearly,
\begin{align}
	f = g_q \circ \ldots g_0 = h_q \circ \ldots \circ h_0.
	\label{eq.f=h_comp}
\end{align}
Using the definition of the H\"older balls $C_r^\beta(D,K),$ it follows that $h_{0j}$ takes values in $[0,1],$ $h_{0j} \in \mC_{t_0}^{\beta_0}([0,1]^{t_0},1),$  $h_{ij} \in  \mC_{t_i}^{\beta_i}([0,1]^{t_i},(2K_{i-1})^{\beta_i})$ for $i=1, \ldots, q-1,$ and $h_{qj}  \in  \mC_{t_q}^{\beta_q}([0,1]^{t_q},K_q(2K_{q-1})^{\beta_q}).$ Without loss of generality, we can always assume that the radii of the H\"older balls are at least one, that is, $K_i \geq 1.$

\begin{lem}
\label{lem.comp_approx}
Let $h_{ij}$ be as above with $K_i \geq 1$. Then, for any functions $\widetilde h_i = (\widetilde h_{ij})_j^\top$ with $\widetilde h_{ij}:[0,1]^{t_i}\rightarrow [0,1],$
\begin{align*}
	&\big \| h_q \circ \ldots  \circ h_0 - \widetilde h_q \circ \ldots  \circ \widetilde h_0 \big\|_{L^\infty[0,1]^d} \\
	&\leq K_q \prod_{\ell=0}^{q-1} (2K_\ell)^{\beta_{\ell+1}}
	\sum_{i=0}^q \big\| |h_i - \widetilde h_i |_\infty \big\|_{L^\infty[0,1]^{d_i}}^{\prod_{\ell = i+1}^q \beta_\ell \wedge 1}.
\end{align*}
\end{lem}

\begin{proof}
Define $H_i = h_i \circ \ldots \circ h_0$ and $\widetilde H_i = \widetilde h_i \circ \ldots \circ \widetilde h_0.$ If $Q_i$ is an upper bound for the H\"older semi-norm of $h_{ij},$ $j=1,\ldots, d_{i+1},$ we find using triangle inequality, 
\begin{align*}
	&\big | H_i(\bx) - \widetilde H_i(\bx)\big|_\infty \\
	&\leq | h_i \circ H_{i-1}(\bx) -  h_i \circ \widetilde H_{i-1}(\bx)\big|_\infty
	+ | h_i \circ \widetilde H_{i-1}(\bx) - \widetilde h_i \circ \widetilde H_{i-1}(\bx)\big|_\infty\\
	&\leq Q_i \big| H_{i-1}(\bx) - \widetilde H_{i-1}(\bx) \big|_\infty^{\beta_i\wedge 1} + \| |h_i-\widetilde h_i|_\infty \|_{L^\infty[0,1]^{d_i}}.
\end{align*}
Together with the inequality $(y+z)^\alpha \leq y^\alpha + z^\alpha $ which holds for all $y, z \geq 0$ and all $\alpha \in [0,1],$ the result follows.
\end{proof}

\begin{proof}[Proof of Theorem \ref{thm.main}]
It is enough to prove the result for all sufficiently large $n.$ Throughout the proof $C'$ is a constant only depending on $(q, \bd, \bt, \bbeta, F)$ that may change from line to line. Combining Theorem \ref{thm.oracle_ineq} with the assumed bounds on the depth $L$ and the network sparsity $s,$ it follows for $n \geq 3,$
\begin{align}
	\begin{split}
	&\frac 14 \Delta_n(\widehat f_n, f_0) - C'\phi_n  L \log^2 n
	\leq 
	R(\widehat f, f_0) \\
	&\leq 4 \inf_{f^*\in \mF(L,\bp,s, F)} \big\| f^* - f_0\big\|_\infty^2 + 4\Delta_n(\widehat f_n, f_0) + C'\phi_n L  \log^2 n,
	\end{split}
	\label{eq.proof_mt0}
\end{align}
where we used $\eps =1/2$ for the lower bound and $\eps =1$ for the upper bound. Taking $C=8C',$ we find that $\tfrac 18 \Delta_n(\widehat f_n, f_0) \leq R(\widehat f, f_0)$ whenever $\Delta_n(\widehat f_n, f_0)\geq C\phi_n L \log^2 n.$ This proves the lower bound in \eqref{eq.main2}.

To derive the upper bounds in \eqref{eq.main1} and \eqref{eq.main2} we need to bound the approximation error. To do this, we rewrite the regression function $f_0$ as in \eqref{eq.f=h_comp}, that is,  $f_0=h_q \circ \ldots h_0$ with $h_i=(h_{ij})_j^\top,$ $h_{ij}$ defined on $[0,1]^{t_i},$ and for any $i<q,$ $h_{ij}$ mapping to $[0,1].$

We apply Theorem \ref{thm.approx_network_One_fct} to each function $h_{ij}$ separately. Take $m = \lceil \log_2 n \rceil$ and let $L'_i:= 8+(\lceil \log_2 n \rceil+5)(1+\lceil \log_2 (t_i\vee \beta_i) \rceil).$ This means there exists a network $\widetilde h_{ij} \in \mF(L'_i, (t_i, 6(t_i+\lceil \beta_i\rceil )N,\ldots,6(t_i+\lceil \beta_i\rceil ) N,1), s_i)$ with $s_i \leq 141 (t_i+\beta_i+1)^{3+t_i} N (\lceil \log_2 n \rceil+6),$ such that
\begin{align}
	\| \widetilde h_{ij} - h_{ij}\|_{L^\infty([0,1]^{t_i})}\leq (2Q_i+1)(1+t_i^2 +\beta_i^2)6^{t_i} N  n^{-1}+ Q_i 3^{\beta_i} N^{-\frac{\beta_i}{t_i}},
	\label{eq.proof_mt1}
\end{align}
where $Q_i$ is any upper bound of the H\"older norms of $h_{ij}.$ If $i <q,$ then we apply to the output the two additional layers $1-(1-x)_+.$ This requires four additional parameters. Call the resulting network $h_{ij}^* \in \mF(L'_i+2, (t_i, 6(t_i+\lceil \beta_i\rceil )N,\ldots, 6(t_i+\lceil \beta_i\rceil )N,1), s_i+4)$ and observe that $\sigma(h_{ij}^*) = (\widetilde h_{ij}(x)\vee 0) \wedge 1.$ Since $h_{ij}(\bx) \in [0,1],$ we must have 
\begin{align}
	\| \sigma(h_{ij}^*) - h_{ij}\|_{L^\infty([0,1]^r)}\leq \| \widetilde h_{ij} - h_{ij}\|_{L^\infty([0,1]^r)}.
	\label{eq.proof_mt2}
\end{align}
If the networks $h_{ij}^*$ are computed in parallel, $h_i^* =(h_{ij}^*)_{j = 1, \ldots, d_{i+1}}$ lies in the class
\begin{align*}
	 \mF \big(L'_i+2, (d_i, 6r_i N,\ldots,6r_i N,d_{i+1}), d_{i+1}(s_i+4) \big),
\end{align*}
with $r_i:=\max_i d_{i+1}(t_i+\lceil \beta_i\rceil ).$ Finally, we construct the composite network $f^* = \widetilde h_{q1} \circ \sigma(h_{q-1}^*) \circ \ldots \circ \sigma(h_{0}^*)$ which by the composition rule in Section \ref{sec.embedding_props_NNs} can be realized in the class
\begin{align}
	\mF \Big( E, (d,  6r_i N,\ldots,6 r_i N,1), \sum_{i=0}^q d_{i+1}(s_i+4) \Big),
	\label{eq.network567}
\end{align}
with $E:=3(q-1)+\sum_{i=0}^q L'_i.$ Observe that there is an $A_n$ that is bounded in $n$ such that $E= A_n + \log_2 n (\sum_{i=0}^q \lceil \log_2 (t_i \vee \beta_i) \rceil +1).$ Using that $\lceil x \rceil < x+1,$ $E \leq \sum_{i=0}^q (\log_2(4)+\log_2 (t_i  \vee \beta_i) ) \log_2 n \leq L$ for all sufficiently large $n.$ 
By \eqref{eq.add_layers} and for sufficiently large $n,$ the space \eqref{eq.network567} can be embedded into $\mF(L, \bp, s)$ with $L, \bp, s$ satisfying the assumptions of the theorem by choosing $N=\lceil c \max_{i=0, \ldots,q} n^{\frac{t_i}{2\beta_i^*+t_i}} \rceil$ for a sufficiently small constant $c>0$ only depending on $q, \bd, \bt, \bbeta.$ Combining Lemma \ref{lem.comp_approx} with \eqref{eq.proof_mt1} and \eqref{eq.proof_mt2} 
\begin{align}
	\inf_{f^*\in \mF(L,\bp,s)} \big\| f^* - f_0\big\|_\infty^2 \leq C' \max_{i=0, \ldots,q} N^{-\frac{2\beta_i^*}{t_i}} \leq C'\max_{i=0, \ldots,q} c^{-\frac{2\beta_i^*}{t_i}}  n^{-\frac{2\beta_i^*}{2\beta_i^*+t_i}}.
	\label{eq.proof_mt3}
\end{align}
For the approximation error in \eqref{eq.proof_mt0} we need to find a network function that is bounded in sup-norm by $F.$ By the previous inequality there exists a sequence $(\widetilde f_n)_n$ such that for all sufficiently large $n,$ $\widetilde f_n \in \mF(L,\bp,s)$ and  $\| \widetilde f_n - f_0\|_\infty^2\leq 2C\max_{i=0, \ldots,q} c^{-2\beta_i^*/t_i}  n^{-(2\beta_i^*)/(2\beta_i^*+t_i)}.$ Define $f^*_n = \widetilde f_n (\|f_0\|_\infty/\| \widetilde f_n \|_\infty \newline \wedge 1).$ Then, $\|f^*_n\|_\infty \leq \|f_0\|_\infty = \|g_q\|_\infty \leq K \leq F,$ where the last inequality follows from assumption $(i).$ Moreover, $f^*_n \in \mF(L, \bp, s, F).$ Writing $f_n^* - f_0 = (f_n^*- \widetilde f_n) + ( \widetilde f_n -f_0),$ we obtain $\|f_n^* - f_0\|_\infty \leq 2 \|\widetilde f_n - f_0\|_\infty.$ This shows that \eqref{eq.proof_mt3} also holds (with constants multiplied by $8$) if the infimum is taken over the smaller space $\mF(L, \bp, s, F).$ Together with \eqref{eq.proof_mt0} the upper bounds in \eqref{eq.main1} and \eqref{eq.main2} follow for any constant $C$. This completes the proof. 
\end{proof}

\subsection{Proof of Theorem \ref{thm.oracle_ineq}}
\label{subsec.proofs.thm.oracle_ineq}

Several oracle inequalities for the least-squares estimator are know, cf. \cite{gyorfi2002, koltchinskii2006, gine2006, hamers2006, massart2007}. The common assumption of bounded response variables is, however, violated in the nonparametric regression model with Gaussian measurement noise. Additionally we provide also a lower bound of the risk and give a proof that can be easily generalized to arbitrary noise distributions. Let $\mN(\delta, \mF, \|\cdot \|_\infty)$ be the covering number, that is, the minimal number of $\|\cdot \|_\infty$-balls with radius $\delta$ that covers $\mF$ (the centers do not need to be in $\mF$).

\begin{lem}
\label{lem.oracle_gen}
Consider the $d$-variate nonparametric regression model \eqref{eq.mod} with unknown regression function $f_0.$ Let $\widehat f$ be any estimator taking values in $\mF.$ Define 
\begin{align*}
	\Delta_n:=\Delta_n(\widehat f, f_0,\mF):= E_{f_0}\Big[\frac 1n \sum_{i=1}^n (Y_i - \widehat f(\bX_i) )^2 -\inf_{f\in \mF} \frac 1n \sum_{i=1}^n (Y_i -f(\bX_i) )^2\Big]
\end{align*}
and assume $\{f_0\}\cup\mF \subset \{f :[0,1]^d \rightarrow [-F,F]\}$ for some $F\geq 1.$ If $\mN_n :=  \mN(\delta, \mF, \|\cdot\|_\infty) \geq 3,$ then, 
\begin{align*}
	&(1-\eps)^2 \Delta_n - F^2\frac{18\log \mN_n + 76}{n\eps} - 38\delta F \\
	&\leq R( \widehat f, f_0) \\
	&\leq (1+\eps)^2 \Big[ \inf_{f\in \mF} E\big[\big(f(\bX) - f_0(\bX) \big)^2 \big] 
	+ F^2 \frac{ 18\log \mN_n + 72}{n\eps } + 32 \delta F +\Delta_n \Big], 
\end{align*}
for all $\delta, \eps \in (0,1].$
\end{lem}

The proof of the lemma can be found in the supplement. Next, we prove a covering entropy bound. Recall the definition of the network function class $\mF(L, \bp, s, F)$ in \eqref{eq.defi_bd_sparse_para_space}. 

\begin{lem}
\label{lem.entropy}
If $V:=\prod_{\ell = 0}^{L+1} (p_\ell +1),$ then, for any $\delta>0,$
\begin{align*}
	\log \mathcal{N}\Big( \delta, \mF(L, \bp, s, \infty), \|\cdot \|_\infty \Big)
	\leq (s+1) \log \Big( 2\delta^{-1} (L+1) V^2 \Big).
\end{align*}
\end{lem}
For a proof see the supplement. A related result is Theorem 14.5 in \cite{anthony1999}. 

\begin{rem}
\label{rem.better_bound_entropy}
Identity \eqref{eq.removal_nodes_identity} applied to Lemma \ref{lem.entropy} yields,
\begin{align*}
	\log \mathcal{N}\Big( \delta, \mF(L, \bp, s, \infty), \|\cdot \|_\infty \Big)
	\leq (s+1) \log \Big( 2^{2L+5} \delta^{-1} (L+1) p_0^2p_{L+1}^2 s^{2L} \Big).
\end{align*}
\end{rem}

\begin{proof}[Proof of Theorem \ref{thm.oracle_ineq}]
The assertion follows from Lemma \ref{lem.entropy} with $\delta =1/n,$ Lemma \ref{lem.oracle_gen} and Remark \ref{rem.better_bound_entropy} since $F\geq 1.$
\end{proof}

\subsection{Proof of Theorem \ref{thm.lb}}

Throughout this proof, $\|\cdot \|_2=\|\cdot \|_{L^2[0,1]^d}.$ By assumption there exist positive $\gamma \leq \Gamma$ such that the Lebesgue density of $\bX$ is lower bounded by $\gamma$ and upper bounded by $\Gamma$ on $[0,1]^d.$ For such design, $R(\widehat f_n, f_0) \geq \gamma \|\widehat f_n - f_0\|_2^2.$ Denote by $P_f$ the law of the data in the nonparametric regression model \eqref{eq.mod}. For the Kullback-Leibler divergence we have $\KL(P_f, P_g)= n E[(f(\bX_1) -g(\bX_1))^2] \leq \Gamma n \|f-g\|_2^2.$ Theorem 2.7 in \cite{tsybakov2009} states that if for some $M\geq 1$ and $\kappa>0,$ $f_{(0)}, \ldots, f_{(M)} \in \mG(q, \bd, \bt, \bbeta, K)$ are such that 
\begin{compactitem}
\item[(i)] $\|f_{(j)}-f_{(k)}\|_2\geq \kappa \sqrt{\phi_n}$ for all $0\leq j < k\leq M$
\item[(ii)] $n\sum_{j=1}^M \|f_{(j)} - f_{(0)}\|_2^2 \leq M\log (M)/(9\Gamma),$
\end{compactitem}
then there exists a positive constant $c=c(\kappa, \gamma),$ such that 
\begin{align*}
	\inf_{\widehat f_n} \, \sup_{f_0 \in \mG(q, \bd, \bt, \bbeta, K)} R\big( \widehat f_n, f_0\big) \geq c \phi_n.
\end{align*} 
In a next step, we construct functions $f_{(0)}, \ldots, f_{(M)} \in \mG(q, \bd, \bt, \bbeta, K)$ satisfying $(i)$ and $(ii).$ Define $i^* \in \argmin_{i=0, \ldots, q } \beta_i^*/(2\beta_i^*+ t_i).$ The index $i^*$ determines the estimation rate in the sense that $\phi_n = n^{-2\beta_{i^*}^*/(2\beta_{i^*}^*+ t_{i^*})}.$ For convenience, we write $\beta^*:= \beta_{i^*},$ $\beta^{**}:= \beta_{i^*}^*,$ and $t^*:=t_{i^*}.$ One should notice the difference between $\beta^*$ and $\beta^{**}.$ Let $K \in L^2(\mathbb{R}) \cap \mC_1^{\beta^*}(\mathbb{R},1)$ be supported on $[0,1].$  It is easy to see that such a function $K$ exists. Furthermore, define $m_n := \lfloor \rho n^{1/(2\beta^{**}+t^*)} \rfloor$ and $h_n :=1/m_n$ where the constant $\rho$ is chosen such that $nh_n^{2\beta^*+t^*} \leq 1/(72 \Gamma \|K^B\|_2^{2t^*})$ with $B:= \prod_{\ell=i^*+1}^q (\beta_\ell\wedge 1).$ For any $\bu = (u_1, \ldots, u_{t^*}) \in \mathcal U_n :=\{(u_1, \ldots, u_{t^*}) : u_i \in \{0, h_n, 2h_n, \ldots, (m_n-1)h_n\},$ define
\begin{align*}
	\psi_{\bu}(x_1, \ldots, x_{t^*}) := h_n^{\beta^*}\prod_{j=1}^{t^*} K\Big(\frac{x_j-u_j}{h_n}\Big).
\end{align*}
For $\balpha$ with $|\balpha|<\beta^*,$ we have $\| \partial^{\balpha} \psi_{\bu} \|_\infty \leq 1$ using $K \in \mC_1^{\beta^*}(\mathbb{R},1).$ For $\balpha=(\alpha_1,\ldots, \alpha_{t^*})$ with $|\balpha|=\lfloor \beta^*\rfloor,$ triangle inequality and $K \in \mC_1^{\beta^*}(\mathbb{R},1)$ gives
\begin{align*}
	h_n^{\beta^*-\lfloor \beta^*\rfloor} \frac{|\prod_{j=1}^{t^*} K^{(\alpha_j)}(\tfrac{x_j-u_j}{h_n}) - \prod_{j=1}^{t^*} K^{(\alpha_j)}(\tfrac{y_j-u_j}{h_n}) |}{\max_{i} |x_i-y_i|^{\beta^*-\lfloor \beta^*\rfloor }}\leq t^*. 
\end{align*}
Hence $\psi_{\bu} \in \mC_{t^*}^{\beta^*}([0,1]^{t^*}, (\beta^*)^{t^*} t^*).$ For a vector $\bw=(w_{\bu})_{\bu \in \mathcal U_n} \in \{0,1\}^{|\mathcal U_n|},$ define 
\begin{align*}
	\phi_{\bw} = \sum_{\bu \in \mathcal U_n} w_{\bu} \psi_{\bu}.
\end{align*}
By construction, $\psi_{\bu}$ and $\psi_{\bu'},$ $\bu,\bu'\in \mathcal U_n ,$ $\bu \neq \bu'$ have disjoint support. As a consequence $\phi_{\bw} \in \mC_{t^*}^{\beta^*}([0,1]^{t^*}, 2(\beta^*)^{t^*} t^*).$

For $i< i^*,$ let $g_i(\bx):= (x_1, \ldots, x_{d_i})^\top.$ For $i=i^*$ define $g_{i^*, \bw}(\bx) = (\phi_{\bw}(x_1, \ldots, x_{t_{i^*}}), 0,\ldots, 0)^\top.$ For $i>i^*,$ set $g_i(\bx) := (x_1^{\beta_i \wedge 1}, 0, \ldots, 0)^\top.$ Recall that $B = \prod_{\ell=i^*+1}^q (\beta_\ell\wedge 1).$ We will frequently use that $\beta^{**} =\beta^* B.$ Because of $t_i \leq \min(d_0, \ldots, d_{i-1})$ and the disjoint supports of the $\psi_{\bu},$  
\begin{align*}
	f_{\bw}(\bx)
	& =g_q \circ \ldots \circ g_{i^*+1} \circ g_{i*,\bw} \circ g_{i^*-1} \circ \ldots \circ g_0(\bx) \\
	& = \phi_{\bw}(x_1, \ldots, x_{t_{i^*}})^B\\
	& =\sum_{\bu \in \mathcal U_n}  w_{\bu} \psi_{\bu}(x_1, \ldots, x_{t_{i^*}})^B
\end{align*}
and $f_{\bw} \in \mG(q, \bd, \bt, \bbeta, K)$ provided $K$ is taken sufficiently large. 

For all $\bu,$ $\|\psi_{\bu} \|_2^2 = h_n^{2\beta^{**} +t^*} \|K^B\|_2^{2t^*}.$ If $\Ham(\bw,\bw')=\sum_{\bu \in \mathcal U_n} \mathbf{1}(w_{\bu} \neq w_{\bu'})$ denotes the Hamming distance, we  find $$\|f_{\bw} - f_{\bw'}\|_2^2 = \Ham(\bw,\bw') h_n^{2\beta^{**} +t^*} \|K^B\|_2^{2t^*}.$$ By the 
Varshamov - Gilbert bound (\cite{tsybakov2009}, Lemma 2.9) and because of $m_n^{t^*} = |\mathcal U_n|,$ we conclude that there exists a subset $\mW \subset \{0,1\}^{m_n^{t^*} }$ of cardinality $| \mW| \geq 2^{m_n^{t^*}/8},$ such that $\Ham(\bw,\bw') \geq m_n^{t^*} /8$ for all $\bw, \bw' \in \mW,$ $\bw \neq \bw'.$ This implies that for $\kappa = \|K^B\|_2^{t^*}/(\sqrt{8}\rho^{\beta^{**}}),$
\begin{align*}
	\|f_{\bw} - f_{\bw'}\|_2^2 \geq \frac 18  h_n^{2\beta^{**}} \|K^B\|_2^{2t^*} \geq \kappa^2 \phi_n \quad \text{for all} \ \bw, \bw' \in \mW, \ \ \bw \neq  \bw'.
\end{align*}
Using the definition of $h_n$ and $\rho,$
\begin{align*}
	n\|f_{\bw} - f_{\bw'}\|_2^2 \leq n m_n^{t^*} h_n^{2\beta^{**}+t^*} \|K^B\|_2^{2t^*} \leq \frac{m_n^{t^*}}{72 \Gamma} \leq \frac{\log |\mW|}{9\Gamma}.
\end{align*}
This shows that the functions $f_{\bw}$ with $\bw \in \mW$ satisfy $(i)$ and $(ii).$ The assertion follows. \qed

\subsection{Proof of Lemma \ref{lem.approx}}

We will choose $c_2 \leq 1.$ Since $\|f_0\|_\infty\leq K,$ it is therefore enough to consider the infimum over $\mF(L,\bp, s, F)$ with $F=K+1.$ Let $\widetilde f_n$ be an empirical risk minimizer. Recall that $\Delta_n(\widetilde f_n, f_0) =0.$ Because of the minimax lower bound in Theorem \ref{thm.lb}, there exists a constant $c_3$ such that $c_3n^{-2\beta/(2\beta+d)}\leq \sup_{f_0 \in \mC_1^\beta([0,1], K)} R(\widetilde f_n, f_0)$ for all sufficiently large $n.$ Because of $p_0=d$ and $p_{L+1}=1,$ Theorem \ref{thm.oracle_ineq} yields 
\begin{align*}
	&c_3n^{-2\beta/(2\beta+d)}
	\leq \sup_{f_0 \in \mC_d^\beta([0,1], K)} R(\widetilde f_n, f_0) \\
	&\leq 
	4\sup_{f_0 \in \mC_d^\beta([0,1], K)} \, \inf_{f\in \mF(L, \bp, s, K+1)} \big\| f - f_0\big\|_\infty^2 + C (K+1)^2\frac{(s+1) \log (n(s+1)^{L}d)}{n}
\end{align*}
for some constant $C.$ Given $\eps,$ set $n_\eps := \lfloor (\sqrt{8}\eps/\sqrt{c_3})^{-(2\beta+d)/\beta}\rfloor.$ Observe that for $\eps\leq \sqrt{c_3/8},$ $n_\eps^{-1}\leq 2 (\sqrt{8}\eps/\sqrt{c_3})^{(2\beta+d)/\beta}$ and $8\eps^2/c_3 \leq n_\eps^{-2\beta/(2\beta+d)}.$ For sufficiently small $c_2>0$ and all $\eps \leq c_2,$ we can insert $n_\eps$ in the previous inequality and find
\begin{align*}
	8 \eps^2 
	\leq 4\sup_{f_0 \in \mC_d^\beta([0,1], K)} \, \inf_{f\in \mF(L, \bp, s, K+1)} \big\| f - f_0\big\|_\infty^2
	+ C_1 \eps^{\frac{2\beta+d}{\beta}} s \big(\log(\eps^{-1}s^L) + C_2\big)
\end{align*}
for constants $C_1, C_2$ depending on $K, \beta,$ and $d.$ The result follows using the condition $s\leq c_1\eps^{-d/\beta}/(L\log(1/\eps))$ and choosing $c_1$ small enough. \qed

\subsection{Proofs for Section \ref{sec.wavelets}}

\begin{proof}[Proof of Lemma \ref{lem.wav_decay}]
 Denote by $r$ the smallest positive integer such that $\mu_r:=\int x^r \psi(x) dx \neq 0.$ Such an $r$ exists because $\{x^r: r=0,1,\ldots\}$ spans $L^2[0,A]$ and $\psi$ cannot be constant. If $h \in L^2(\mathbb{R})$, then we have for the wavelet coefficients
\begin{align}
	&\int h(x_1+\ldots + x_d) \prod_{\ell=1}^d \psi_{j,k_\ell}(x_\ell) \, d\bx \notag \\
	&= 2^{-\frac{jd}2} \int_{[0,2^q]^d} h\Big(2^{-j}\Big(\sum_{\ell=1}^d x_\ell+k_\ell\Big)\Big)  \prod_{\ell=1}^d \psi(x_\ell) \, d\bx.
	\label{eq.wav_coeff_identity}
\end{align}
For a real number $u,$ denote by $\{u\}$ the fractional part of $u.$ 

We need to study the cases $\mu_0\neq 0$ and $\mu_0= 0$ separately. If $\mu_0\neq 0,$ define $g(u)= r^{-1} u^r\mathbf{1}_{[0,1/2]}(u)+r^{-1} (1-u)^r\mathbf{1}_{(1/2,1]}(u).$ Notice that $g$ is Lipschitz with Lipschitz constant one. Let $h_{j,\alpha}(u) = K 2^{-j\alpha-1}g(\{2^{j-q-\nu} u\})$ with $q$ and $\nu$ as defined in the statement of the lemma. For a $T$-periodic function $u\mapsto s(u)$ the $\alpha$-H\"older semi-norm for $\alpha\leq 1$ can be shown to be $\sup_{u \neq v, |u-v|\leq T} |s(u)-s(v)|/|u-v|^\alpha.$  Since $g$ is $1$-Lipschitz, we have for $u,v$ with $|u-v| \leq 2^{q+\nu-j},$ 
\begin{align*}
	\big | h_{j,\alpha}(u) - h_{j,\alpha}(v) \big| \leq K2^{-j\alpha-1} 2^{j-q-\nu} |u-v| \leq \frac{K}2 |u-v|^\alpha. 
\end{align*}
Since $\|h_{j,\alpha}\|_\infty \leq K/2,$  $h_{j, \alpha} \in \mC_1^\alpha([0,d], K).$ Let $f_{j,\alpha}(\bx)=h_{j,\alpha}(x_1+\ldots + x_d).$ Recall that the support of $\psi$ is contained in $[0,2^q]$ and $2^\nu \geq 2d.$ By definition of the wavelet coefficients, Equation \eqref{eq.wav_coeff_identity}, the definitions of $h_{j, \alpha},$ and using $\mu_r =\int x^r \psi(x) dx,$ we find for $p_1,\ldots, p_d \in \{0,1, \ldots, 2^{j-q-2}-1\},$
\begin{align*}
	&d_{(j,2^{q+\nu} p_1)\ldots (j,2^{q+\nu} p_d)}(f_{j,\alpha})\\
	&= 2^{-\frac{jd}2} \int_{[0,2^q]^d} h_{j,\alpha}\Big(2^{-j}\Big(\sum_{\ell=1}^d x_\ell+2^{q+\nu} p_\ell\Big)\Big) \prod_{\ell=1}^d  \psi(x_\ell) \, d\bx\\
	&= K 2^{-\frac{jd}2-j\alpha-1} \int_{[0,2^q]^d} g\Big(\Big\{\frac{\sum_{\ell=1}^d x_\ell}{2^{q+\nu}}\Big\}\Big) \prod_{\ell=1}^d \psi(x_\ell) \, d\bx\\
	&= r^{-1} 2^{-qr-\nu r-1} K 2^{-\frac{j}2 (2\alpha+d)} \int_{[0,2^q]^d}  (x_1+\ldots +x_d)^r \prod_{\ell=1}^d \psi(x_\ell) \, d\bx\\
	&=  d r^{-1} 2^{-qr-\nu r-1} K \mu_0^{d-1} \mu_r 2^{-\frac{j}2 (2\alpha+d)},
\end{align*}
where we used for the last identity that by definition of $r,$ $\mu_1=\ldots=\mu_{r-1}=0.$

In the case that $\mu_0=0,$ we can take $g(u)= (dr)^{-1} u^{dr}\mathbf{1}_{[0,1/2]}(u)+(dr)^{-1} (1-u)^{dr}\mathbf{1}_{(1/2,1]}(u).$ Following the same arguments as before and using the multinomial theorem, we obtain 
\begin{align*}
	d_{(j,2^{q+\nu} p_1)\ldots (j,2^{q+\nu} p_r)}(f_{j,\alpha})
	= \binom{dr}{r} \frac 1{dr} 2^{-dqr-d\nu r-1} K \mu_r^d 2^{-\frac{j}2 (2\alpha+d)}.
\end{align*}
The claim of the lemma follows. 
\end{proof}

\begin{proof}[Proof of Theorem \ref{thm.wavelet_lb}]
Let $c(\psi, d)$ be as in Lemma \ref{lem.wav_decay}. Choose an integer $j^*$ such that $$\frac 1n \leq c(\psi, d)^2 K^2 2^{-j^*(2\alpha+d)}\leq \frac{2^{2\alpha +d}}n.$$ This means that $2^{j^*} \geq \tfrac 12 ( c(\psi, d)^2K^2 n)^{1/(2\alpha+d)}.$ By Lemma \ref{lem.wav_decay}, there exists a function $f_{j^*,\alpha}$ of the form $h(x_1 + \ldots + x_d),$ $h \in \mC_1^\alpha([0,d],K),$ such that with \eqref{eq.wav_est_risk_lb},
\begin{align*}
	R(\widehat f_n, f_{j^*,\alpha}) 
	\geq \sum_{p_1, \ldots, p_d \in \{0,1, \ldots, 2^{j^*-q-\nu}-1\}} \frac{1}{n}
	= \frac 1n 2^{j^*d-qd-\nu d} \gtrsim n^{-\frac{2\alpha}{2\alpha+d}}.
\end{align*}
\end{proof}

\section*{Acknowledgments}
The author is grateful to all the insights, comments and suggestions that arose from discussions on the topic with other researchers. In particular, he wants to thank the AE, two referees, Thijs Bos, Hyunwoong Chang, Konstantin Eckle, Kenji Fukumizu, Maximilian Graf, Roy Han, Masaaki Imaizumi, Michael Kohler, Matthias L\"offler, Patrick Martin, Hrushikesh Mhaskar, Gerrit Oomens, Tomaso Poggio, Richard Samworth, Taiji Suzuki, Dmitry Yarotsky and Harry van Zanten.

\newpage 

\setcounter{section}{0}

\begin{frontmatter}

\title{Rejoinder to discussions of "Nonparametric regression using deep neural networks with ReLU activation function"}
\runtitle{Nonparametric regression using ReLU networks}


\author{\fnms{Johannes} \snm{Schmidt-Hieber}}
\address{University of Twente,\\ P.O. Box 217, \\ 7500 AE Enschede \\ The Netherlands\\ \printead{e1}}
\affiliation{University of Twente}

\runauthor{J. Schmidt-Hieber}



\end{frontmatter}

The author is very grateful to the discussants for sharing their viewpoints on the article. The discussant contributions highlight the gaps in the theoretical understanding and outline many possible directions for future research in this area. The rejoinder is structured according to topics. We refer to [GMMM], [K], [KL], and [S] for the discussant contributions by Ghorbani et al., Kutyniok, Kohler \&  Langer, and Shamir, respectively. 

\section{Overparametrization and implicit regularization}


One of the general claims about deep learning is that even for extreme overfitting the method still generalizes well. There are numerous experiments showing that running the training error to zero and therefore interpolating all data points results in state of the art generalization performance. The rationale behind this is that among all solutions interpolating the data points - of which most result in bad generalization behavior - stochastic gradient descent (SGD) picks a minimum norm interpolant. This is also known as implicit regularization. While this is well known for stochastic gradient descent applied to linear regression, for deep networks some progress has been made recently in finding the norm minimized by (S)GD, see \cite{2018arXiv180208246G, Soudry2019}. 

It is now reasonable to wonder whether the notion of network sparsity could be removed in the article if implicit regularization would have been taken into account. [GMMM] write that "model complexity is not controlled by an explicit penalty or procedure, but by the dynamics of stochastic gradient descent (SGD) itself." [S] mentions implicit regularization to show that statistical guarantees should involve specific learning methods. 

We conjecture that for additive error models, such as the nonparametric regression model considered in the article, implicit regularization in the overfitted regime is insufficient to achieve even consistency. To support our conjecture, we provide the following two step argument. In the first step, we argue that for one-dimensional input and shallow networks with fixed parameters in the first layer, SGD will converge to a variant of the natural cubic spline interpolant. In the second step we show that this reconstruction leads to an inconsistent estimator if additive noise is present. 


A shallow ReLU network with one input and one output node can be written as $x\mapsto \sum_{j=1}^m a_j(b_jx-c_j)_+.$ We now study an even more simplified setup where $b_j$ is always one. For small $\delta >0,$ $(x-c_j)_+ \approx \int_{c_j}^{c_j+\delta} (x-u)_+du/\delta.$ This motivates to study smoothed shallow ReLU networks of the form 
\begin{align*}
	x\mapsto f_{\ba}(x)=\sum_{j=1}^m \frac{a_j}{\sqrt{t_j-t_{j-1}}} \int_{t_{j-1}}^{t_j} (x- u)_+du.
\end{align*}
with parameter vector $\ba = (a_1, \ldots, a_m)$ and fixed $t_0< t_1 < \ldots < t_m.$ For convenience, we have rescaled the parameters $a_j$ so that the normalization factor becomes $1/\sqrt{t_j-t_{j-1}}.$ We consider the overparametrized regime $m\geq n$ assuming that for any $i,$ there lies at least one $t_j$ in the interval $[X_{(i-1)},X_{(i)})$ with $X_{(i)}$ the $i$-th order statistic of the sample $X_1,\ldots,X_n$ and $X_{(0)}=-\infty.$ Under overparametrization, this is a rather weak assumption and ensures existence of a shallow ReLU network $f_{\ba}^*$ perfectly interpolating the data in the sense that $f_{\ba}^*(X_i)=Y_i$ for all $i.$

For initialization at zero and properly chosen learning rate, SGD with respect to the least squares loss converges to the minimum norm interpolant with parameter vector
\begin{align*}
	\ba^* = \argmin_{\ba \in \R^m}\big\{ \| \ba \|_2: f_{\ba}(X_i)=Y_i, \ \forall i \big\}
\end{align*}
(this result is due to \cite{MR2500924} for overdetermined linear systems but can be extended to the underdetermined case, see also the generalizations in \cite{gower2015stochastic,MR3432148, 2018arXiv180208246G}). Because of $f_{\ba}''(x)=a_j /\sqrt{t_j-t_{j-1}}$ for all $x\in (t_{j-1}, t_j),$ we find $\| \ba \|_2 = \|f_{\ba}''\|_{L^2[t_0,t_m]}.$ It is known that the natural cubic spline interpolant $L$ is the interpolant with the smallest $L^2$-norm on the second derivative. Moreover, we have that $\|f''\|_{L^2}^2 = \|L''\|_{L^2}^2+ \|L''-f''\|_{L^2}^2$ for all twice differentiable interpolating functions $f$, see Equation (2.9) in \cite{MR1270012}. Since $f_{\ba^*}$ and $L$ are both interpolants, this implies that the SGD limit $f_{\ba^*}$ will be close to the natural cubic spline interpolant.

In the nonparametric regression model with additive errors, the distance between the true function values and the response variables $Y_i$ is of the order of the noise level (which is assumed to be fixed). The natural cubic spline interpolates the $Y_i$'s. If in a neighborhood, the $Y_i$'s lie all on one side of the regression function, the average distance between the natural cubic spline interpolant and the true regression function will be lower bounded by a constant. Since this happens on a subset with Lebesgue measure bounded from below, the natural cubic spline interpolant is inconsistent for estimating the regression function. As the SGD limit approximates the natural cubic spline interpolant, this indicates that stochastic gradient descent should lead to inconsistent estimators.

We believe that this also holds true for deep networks. In this case, it is expected that SGD still converges to a spline interpolant but not necessarily to the natural cubic spline interpolant, see also \cite{2019arXiv190205040S} for a related argument. 

While it has been observed that there are nonparametric estimators that can interpolate and also achieve fast convergence rates in the nonparametric regression model (\cite{BelkinRT2019}), the argument above indicates that implicit regularization in the overfitted regime will not do that. To obtain rate optimal estimators, more regularization has to be imposed forcing the network to do smoothing. 

\section{Network sparsity}

The article identifies sparsity of the network weights as a complexity measure to achieve optimal convergence rates under a hierarchical composition assumption. As sparsity is a non-standard assumption, there are several comments on it in the reports. [GMMM] show that the empirical distribution of the weights in the first fully connected layer of the VGG-19 network is nearly Gaussian. [KL] mention a recent result proving optimal estimation rates for very deep networks with fully connected layers. 

After the original version of this article was drafted, a large body of applied work emerged dealing either with compression through sparsifying dense networks or proposing methods that directly train a sparse neural network. Below we briefly summarize some of these approaches.

One method to achieve sparsity in neural networks is by pruning a fully connected network after training. A simple approach would be to replace small network weights by zero, but more sophisticated approaches based on the second derivative have been proposed as well, see \cite{NIPS1989_250, NIPS1992_647,NIPS2015_5784}. \cite{frankle2018lottery} proposes an iterative pruning procedure, see also \cite{2019arXiv190209574G}. These approaches allow to reduce the number of parameters in fully connected layers by about $90\%$ without loss of efficiency. 


Although Theorem 1 is formulated in terms of network sparsity, the proof explicitly constructs a network topology, that is, the graph structure defined by the non-zero connections between successive layers, for which the minimax estimation rate is attained (up to log-factors). Instead of searching over all $s$-sparse networks, it is therefore in principle possible to start with this network topology and only learn the non-zero weights. By fixing one sparse network topology, a lot of the flexibility of networks to adapt to the underlying structure in the data might be lost. An intermediate constraint would be to impose an individual sparsity parameter for each weight matrix or to bound the indegree and outdegree for each individual unit in the network. In the applied literature, choosing a sparse network topology beforehand has been proposed recently in \cite{Prabhu2018, 2018arXiv180905242R}. The latter article makes an interesting connection between sparsely connected neural networks and decision trees. Related to an initial choice of a sparse network topology is the evolutionary algorithm inspired by biological neural networks proposed in \cite{Mocanu2018}. It starts with sparse weight matrices. In every iteration, the smallest weights are removed and new random connections are added so that the network topology changes but the overall network sparsity is kept constant. The method proposed in \cite{2019arXiv190311257A} is also inspired by the sparsity observed in biological networks. It starts with a sparse network topology and increases the sparsity by only keeping the units in each hidden layer that channel most of the signal to the next layer. 

The recent work \cite{GaierHa2019} on weight agnostic neural networks takes this one step further. No training is done and the weights are fixed to the initialized values at all times. Only the network topology is learned by an iterative procedure. In each step of the iteration, we have a set of candidate models. For each of those models a score is computed. "Around" the models with the highest scores a new set of randomly generated candidate models is generated. 

Theorem 1 in [KL] considers neural networks with fixed width and depth increasing polynomially in the sample size. It is shown that for such extremely deep networks, the empirical risk minimizer over fully-connected layers achieves the optimal estimation rate and no sparsity is needed. Such architectures are, however, in many aspects quite different compared to the neural networks considered in practice. In \cite{yarotsky18a}, it has been observed that for such extremely deep networks, one needs discontinuous weight assignments to achieve the best possible approximation rate. This is a strange phenomenon which could hint at some issues with the stability during learning of the network weights.

\section{Classification and nonparametric regression}

While the article deals with data from the nonparametric regression model, the overwhelming part of the literature on deep learning is on classification. Nonparametric regression and estimation of the conditional class probabilities in classification is similar, if a fraction of the data is mislabeled which prevents the conditional class probabilities to be close to zero or one. For the commonly considered classification tasks in deep learning, this is, however, not the case as most of the data are correctly labeled. As the randomness due to mislabeling is negligible in those cases, the only remaining randomness is in the distribution of the design/inputs and reconstruction becomes rather an interpolation than a denoising problem. If the different classes are also well separated from each other, much faster convergence rates can be achieved. This explains why the sample complexity in the nonparametric regression model is much higher than what is observed in deep learning for object recognition tasks, see also Report [S].

Concerning the statistical properties there are some differences. For image classification problems, deep learning is for instance not robust to Gaussian perturbations, see \cite{hendrycks2018benchmarking}. In the nonparametric regression model, Gaussian perturbations just increase the noise level. Since the noise level  appears in the estimation risk bounds through the constants, the estimation rates for the class of estimators considered in the article will not change under additive noise perturbations. 

We would like to stress again that the structure of the data is essential for the behavior of deep learning and the properties of the reconstructions. One of the challenges for future research will be to study estimation in models beyond nonparametric regression.

\section{Algorithms}

[GMMM] and [S] question whether one can disentangle the algorithm from the statistical analysis. We would like to stress that Theorem 1 is not about one fixed estimator. It provides bounds for any estimator which, given data, returns a sparsely connected neural network. The method/estimator determines the term $\Delta_n(\widehat f_n,f_0)$ defined in Equation (5) and Theorem 1 shows that $\Delta_n(\widehat f_n,f_0)$ tightly controls the estimation risk from above and below. This is different than the case of data interpolation and training error zero, where $\Delta_n(\widehat f_n,f_0)$ is not sufficient anymore to fully characterize the statistical properties, see also Report [S] and \cite{2016arXiv161103530Z}. 

We agree that the difficulty is shifted to a precise estimate of the term $\Delta_n(\widehat f_n,f_0)$ and we hope to study this term in more detail in future work. This term might heavily depend on the learning rate, the initialization and the energy landscape. Regarding a question in [K], the expectation in the definition of $\Delta_n(\widehat f_n,f_0)$ (Equation (5) in the article) can be taken over all the randomness, including additional randomization in the algorithm. 

While it would be desirable to have precise theoretical bounds for the performance of the most popular deep learning methods such as Adam, we believe that some amount of idealization and simplification is unavoidable. In statistical theory, this seems to be widely accepted. For instance, most of the theory on the LASSO deals with regularization parameters derived from large deviations bounds although the standard software implementations choose the regularization parameter by ten-fold cross validation.


\section{High-dimensional input}

[GMMM] mentions that for the current proof strategy and the case of additive models, the dependence of the dimension on the constants is $d^d.$ As mentioned in the article, the results focus on the convergence rates and no attempt has been made to minimize the constants appearing in the proofs. In fact by a variation of the original argument, the dependence on the dimension for additive models $f(\bx)=\sum_{i=1}^d f_i(x_i)$ can be shown to be linear. To see this, we can build for any given $N\geq 1,$ $d$ separate networks with $s \asymp N \log N$ parameters, computing the functions $f_1(x_1), \ldots, f_d(x_d)$ up to an approximation error of the order $O(N^{-\beta}).$ Using the parallelization rule mentioned on p.21 of the article, one can then combine the individual networks into a large neural network computing the sum $\sum_{i=1}^n f_i(x_i)$ up to an approximation error of order $O(dN^{-\beta})$ using $s \asymp dN\log N$ many network parameters. It then follows from Theorem 2 that the rate is upper bounded by $dn^{-2\beta/(2\beta+1)}\log^3 n$ if $\Delta_n(\widehat f_n,f_0)$ is sufficiently small and $d$ is bounded by a power of the sample size. 

As another result on high-dimensional input, [S] mentions a theorem proving that basis expansions have difficulties to approximate functions generated by a single neuron. Either huge coefficients are needed or the number of basis functions has to be exponential in the input dimension. 

Since the input dimension $d$ in deep learning applications is typically extremely large, a possible future direction would be to analyze neural networks with high-dimensional $d=d_n \uparrow \infty$ and comparing the rates to other nonparametric procedures. 

\section{Function classes}

With respect to the considered function class, [K] emphasizes that the function classes should be detached from the method. On the contrary, [S] favors an alternative approach where the underlying function class consists itself of neural network functions. We believe that both approaches have advantages and disadvantages. 

The imposed class of composition functions in the article appears of course naturally given the composition structure of deep networks. Compositions are fundamental operations and as mentioned in the article, many widely studied function classes in nonparametric statistics such as (generalized) additive models occur as special cases of the imposed composition constraint.

For a recent result in the statistical literature with function class consisting of neural network functions, we refer to \cite{2018arXiv180903090B}. One possibility for future research would be to determine the maxisets for neural networks, that is, the largest possible function class for which a prespecified estimation rate can be obtained, see \cite{MR1895892}. The main advantage of generic function spaces such as H\"older classes is that we can compare the estimation rates achieved by different methods and therefore learn something about the strength and weaknesses of these methods. The article shows for instance that wavelet methods have a slower rate of convergence for generalized additive models than sparsely connected deep ReLU networks.

To obtain fast estimation rates, an alternative is to impose structure on the design, see \cite{2019arXiv190702177N,2019arXiv190800695S}.

\section{Choice of the activation function}

On p. 12 in the article we highlight several specific properties of the ReLU activation function such as the possibility to easily learn skip connections. [KL] mention that results for ReLU networks automatically transfer to other activation functions. The argument, however, requires that the network parameters will become large. In the meantime, we  better and better understand how SGD leads to norm control on the parameters. To model this, we think that it is important to control the magnitude of the weights in the network classes. In the article, the network parameters are bounded in absolute value by one. This is a convenient choice, but as our understanding of the norm control induced by SGD improves, more realistic constraints are imaginable. It is well-known that training does not move the parameters far from the initialized values. To analyze the effect of different initialization strategies one possibility would be to study network classes generated by all parameters in a neighborhood of a (random) initializer.

\section{Real data}

[GMMM] report the results of a simulation study which seemingly contradict the theory in the article. They study the noisefree case and up to three hidden layers showing that a certain smooth function cannot be learned. We would like to refer to the simulation study in \cite{ECKLE2019232}, which finds that for regression problems, the performance of deep neural networks is not far off from the theoretical bounds. This article also examines the finite sample performance of the multiplication network in Lemma A.2 which forms an essential part in the proof of Theorem 1. To a certain extent, even such specific constructions can be picked up by deep learning. This, however, only works for a careful initialization. It might be necessary to reinitialize the procedure if the algorithm gets stuck in a local minimum with large training error.

\section*{Acknowledgments}
The author would like to thank Misha Belkin and Dirk Lorentz for fruitful discussions on overfitting and SGD.

\newpage

\begin{frontmatter}

\title{Supplement to "Nonparametric regression using deep neural networks with ReLU activation function"}
\runtitle{Nonparametric regression using ReLU networks}

\author{\fnms{Johannes} \snm{Schmidt-Hieber}}
\address{University of Twente,\\ P.O. Box 217, \\ 7500 AE Enschede \\ The Netherlands\\ \printead{e1}}
\affiliation{University of Twente}


\runauthor{J. Schmidt-Hieber}

\begin{abstract}
The supplement contains additional material for the article "Nonparametric regression using deep neural networks with ReLU activation function".
\end{abstract}

\end{frontmatter}

\appendix

\section{Network approximation of polynomials}

In this section we describe the construction of deep networks that approximate monomials of the input.

In a first step, we construct a network, with all network parameters bounded by one, which approximately computes $xy$ given input $x$ and $y.$ Let $T^k : [0,2^{2-2k}] \rightarrow [0,2^{-2k}],$
\begin{align*}
	&T^k(x) := (x/2) \wedge (2^{1-2k}-x/2) =  (x/2)_+ - (x -2^{1-2k})_+	
\end{align*}
and $R^k :[0,1]\rightarrow [0,2^{-2k}],$
\begin{align*}
	&R^k:= T^k \circ T^{k-1} \circ \ldots T^1.
\end{align*}
The next result shows that $\sum_{k=1}^m R^k(x)$ approximates  $x(1-x)$ exponentially fast in $m$ and that in particular $x(1-x)= \sum_{k=1}^\infty R^k(x)$ in $L^\infty[0,1].$ This lemma can be viewed as a slightly sharper variation of Lemma 2.4 in \cite{telgarsky2016} and Proposition 2 in \cite{yarotski2017}. In contrast to the existing results, we can use it to build networks with parameters bounded by one. It also provides an explicit bound on the approximation error.


\begin{lem2}
\label{lem.square_approx}
For any positive integer $m,$
\begin{align*}
	\big | x(1-x)- \sum_{k=1}^m R^k(x)  \big | \leq 2^{-m}.
\end{align*}
\end{lem2}

\begin{proof}
In a first step, we show by induction that $R^k$ is a triangle wave. More precisely, $R^k$ is piecewise linear on the intervals $[\ell/2^k,(\ell+1)/2^k]$ with endpoints $R^k(\ell/2^k) = 2^{-2k}$ if $\ell$ is odd and $R^k(\ell/2^k) = 0$ if $\ell$ is even. For $R^1=T^1$ this is obviously true. For the inductive step, suppose this is true for $R^k.$ Write $\ell \equiv r \mod 4$ if $\ell -r$ is divisible by $4$ and consider $x\in [\ell /2^{k+1},(\ell+1)/2^{k+1}].$ If $\ell \equiv 0 \mod 4$ then, $R^k(x) = 2^{-k}(x-\ell/2^{k+1}).$ Similar for $\ell \equiv 2\mod 4,$ $R^k(x) = 2^{-2k} - 2^{-k}(x-\ell/2^{k+1});$ for $\ell \equiv 1 \mod 4,$ we have $\ell+1 \equiv 2 \mod 4$ and $R^k(x) = 2^{-2k-1}+2^{-k} (x-\ell/2^{k+1});$ and for $\ell \equiv 3 \mod 4,$ $R^k(x) =2^{-2k-1} - 2^{-k}(x-\ell/2^{k+1}).$ Together with
\begin{align*}
	R^{k+1}(x) 
	&= T^{k+1}\circ R^k (x) \\
	&= \frac{R^k(x)}  2 \mathbf{1} (R^k(x) \leq 2^{-2k-1})+ \Big( 2^{-2k-1}- \frac{R^k(x)}  2 \Big) \mathbf{1} (R^k(x) > 2^{-2k-1}).
\end{align*}
this shows the claim for $R^{k+1}$ and completes the induction.

For convenience, write $g(x) = x(1-x).$ In the next step, we show that for any $m \geq 1$ and any $\ell \in \{0,1, \ldots, 2^m\},$
\begin{align*}
	g(\ell 2^{-m}) = \sum_{k=1}^m R^k(\ell 2^{-m}).
\end{align*}
We prove this by induction over $m.$ For $m=1$ the result can be checked directly. For the inductive step, suppose that the claim holds for $m.$ If $\ell$ is even we use that $R^{m+1}(\ell 2^{-m-1})=0$ to obtain that $g(\ell 2^{-m-1}) = \sum_{k=1}^m R^k(\ell 2^{-m-1}) = \sum_{k=1}^{m+1} R^k(\ell 2^{-m-1}).$ It thus remains to consider $\ell$ odd. Recall that $x\mapsto \sum_{k=1}^m R^k(x)$ is linear on $[(\ell-1)2^{-m-1}, (\ell+1)2^{-m-1}]$ and observe that for any real $t,$
\begin{align*}
	g(x) - \frac{g(x+t)+g(x-t)}{2} = t^2.
\end{align*}
Using this for $x =\ell 2^{-m-1}$ and $t=2^{-m-1}$ yields for odd $\ell$ due to $R^{m+1}(\ell 2^{-m-1}) \\ =2^{-2m-2},$
\begin{align*}
	g(\ell 2^{-m-1}) = 2^{-2m-2}+\sum_{k=1}^m R^k(\ell 2^{-m-1})  = \sum_{k=1}^{m+1} R^k(\ell 2^{-m-1}).
\end{align*}
This completes the inductive step.

So far we proved that $\sum_{k=1}^m R^k(x)$ interpolates $g$ at the points $\ell 2^{-m}$ and is linear on the intervals $[\ell 2^{-m}, (\ell +1 )2^{-m}].$ Observe also that $g$ is Lipschitz with Lipschitz constant one. Thus, for any $x,$ there exists an $\ell$ such that
\begin{align*}
	\big |g(x) - \sum_{k=1}^m R^k(x) \big| 
	&= \Big| g(x) - (2^m x - \ell) g \big((\ell+1)2^{-m}\big) - (\ell +1 -2^m x) g(\ell 2^{-m}) \Big| \\
	&\leq 2^{-m}.
\end{align*}
\end{proof}

Let $g(x) =x(1-x)$ as in the previous proof. To construct a network which returns approximately $xy$ given input $x$ and $y$, we use the polarization type identity
\begin{align}
	g\Big(\frac{x-y+1}{2}\Big) - g\Big( \frac{x+y}{2} \Big) + \frac{x+y}{2} - \frac 14 = xy,
	\label{eq.polarization}
\end{align}
cf. also \cite{yarotski2017}, Equation (3).

\begin{lem2}
\label{lem.mult}
For any positive integer $m,$ there exists a network $\Mult_m \in \mF(m+4,(2,6,6,\ldots,6,1)),$ such that $\Mult_m(x,y) \in [0,1],$
\begin{align*}
	\big|\Mult_m (x,y) - x y \big| \leq 2^{-m}, \quad \text{for all} \ x,y \in [0,1],
\end{align*}
and $\Mult_m(0,y)=\Mult_m(x,0) =0.$
\end{lem2}

\begin{figure}
\begin{center}
	\includegraphics[scale=0.55]{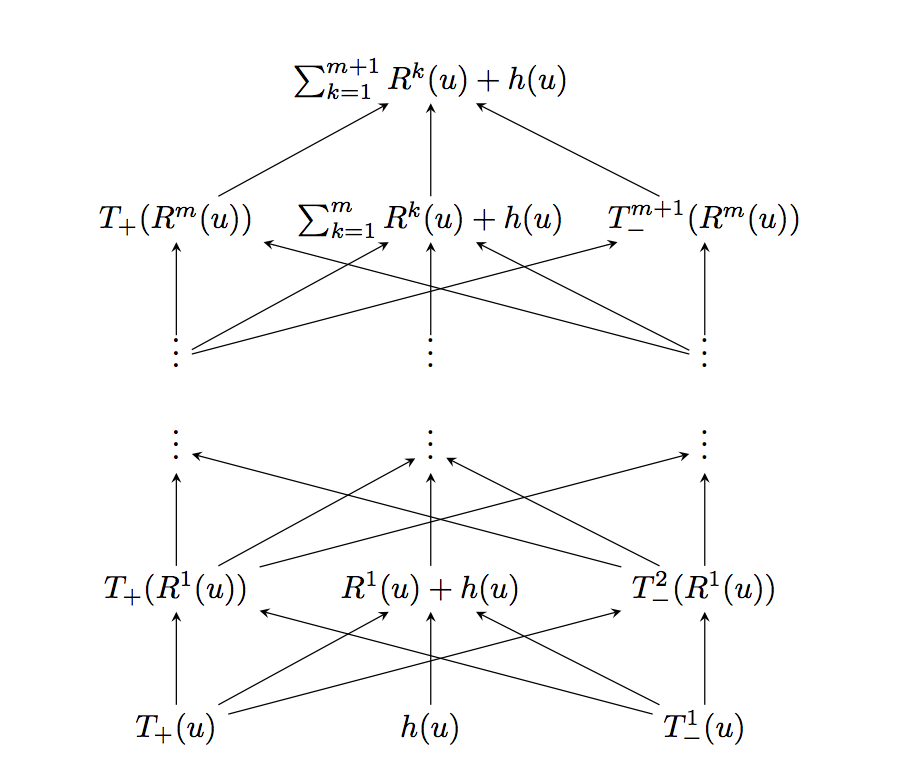} 
\end{center}
\vspace{-0.6cm}
\caption{\label{fig.mult_net} The network $( T_+(u) , T_-^1(u), h(u)) \mapsto \sum_{k=1}^{m+1} R^k(u) + h(u).$}
\end{figure}

\begin{proof}
Write $T_k(x) = (x/2)_+ - (x -2^{1-2k})_+ = T_+(x) -T_-^k(x)$ with $T_+(x) = (x/2)_+$ and $T_-^k(x) = (x -2^{1-2k})_+$ and let $h:[0,1] \rightarrow [0, \infty)$ be a non-negative function. In a first step we show that there is a network $N_m$ with $m$ hidden layers and width vector $(3,3,\ldots,3,1)$ that computes the function
\begin{align*}
	\big( T_+(u) , T_-^1(u), h(u) \big) \mapsto \sum_{k=1}^{m+1} R^k(u) + h(u),
\end{align*}
for all $u \in [0,1].$ The proof is given in Figure \ref{fig.mult_net}. Notice that all parameters in this networks are bounded by one. In a next step, we show that there is a network with $m+3$ hidden layers that computes the function
\begin{align*} 
	(x,y) \mapsto \Big( \sum_{k=1}^{m+1} R^k\Big( \frac{x-y+1}{2}\Big) - \sum_{k=1}^{m+1} R^k\Big( \frac{x+y}{2}\Big) + \frac{x+y}{2}  - \frac 14 \Big)_+ \wedge 1.
\end{align*}
Given input $(x,y),$ this network computes in the first layer
\begin{align*}
	\Big( T_+\Big(\frac{x-y+1}{2}\Big), T_-^1\Big(\frac{x-y+1}{2}\Big), \Big(\frac{x+y}{2}\Big)_+,
	T_+\Big(\frac{x+y}{2}\Big), T_-^1\Big(\frac{x+y}{2}\Big), \frac 14 \Big).
\end{align*}
On the first three and the last three components, we apply the network $N_m.$ This gives a network with $m+1$ hidden layers and width vector $(2,6,\ldots,6,2)$ that computes
\begin{align*}
	(x,y) \mapsto \Big( \sum_{k=1}^{m+1} R^k\Big( \frac{x-y+1}{2}\Big) + \frac{x+y}{2}, \sum_{k=1}^{m+1} R^k\Big( \frac{x+y}{2}\Big) + \frac 14 \Big). 
\end{align*}
Apply to the output the two hidden layer network $(u,v) \mapsto (1 -(1-(u-v))_+)_+ =(u-v)_+\wedge 1.$ The combined network $ \Mult_m(x,y)$ has thus $m+4$ hidden layers and computes 
\begin{align} 
	(x,y) \mapsto \Big( \sum_{k=1}^{m+1} R^k\Big( \frac{x-y+1}{2}\Big) - \sum_{k=1}^{m+1} R^k\Big( \frac{x+y}{2}\Big) + \frac{x+y}{2}  - \frac 14 \Big)_+ \wedge 1.
	\label{eq.output_mult_network}
\end{align}
This shows that the output is always in $[0,1].$ By \eqref{eq.polarization} and Lemma \ref{lem.square_approx}, $|\Mult_m(x,y) -xy| \leq 2^{-m}.$

By elementary computations, one can check that $R^1((1-u)/2)=R^1( (1+u)/2)$ and $R^2((1+u)/2)=R^2(u/2)$ for all $0\leq u \leq 1.$ Therefore, $R^k((1-u)/2)=R^k( (1+u)/2)$ for all $k\geq 1$ and $R^k((1+u)/2)=R^k(u/2)$ for all $k\geq 2.$ This shows that the output \eqref{eq.output_mult_network} is zero for all inputs $(0,y)$ and $(x,0).$
\end{proof}

\begin{lem2}
\label{lem.multi_mult}
For any positive integer $m,$ there exists a network $$\Mult_m^r \in \mF((m+5) \lceil \log_2 r \rceil , (r, 6r,6r, \ldots, 6r,1))$$ such that $\Mult_m^r \in [0,1]$ and
\begin{align*}
	\Big|\Mult_m^r 
	(\bx)
	- \prod_{i=1}^r x_i \Big| \leq r^2 2^{-m}, \quad \text{for all} \ \ \bx=(x_1, \ldots,x_r)\in [0,1]^r.
\end{align*}
Moreover, $\Mult_m^r(\bx)=0$ if one of the components of $\bx$ is zero. 
\end{lem2}

\begin{proof}
Let $q:= \lceil \log_2 (r) \rceil$. Let us now describe the construction of the $\Mult_m^r$ network. In the first hidden layer the network computes 
\begin{align}
	(x_1, \ldots, x_r) \mapsto (x_1, \ldots, x_r, \underbrace{1,\ldots,1}_{2^q-r}).
	\label{eq.1st_layer_mult}
\end{align}
Next, apply the network $\Mult_m$ in Lemma \ref{lem.mult} to the pairs $(x_1,x_2),$ $(x_3,x_4)$,\ldots, $(1,1)$ in order to compute $(\Mult_m(x_1,x_2), \Mult_m(x_3,x_4), \ldots, \Mult_m(1,1)) \in \mathbb{R}^{2^{q-1}}.$  Now, we pair neighboring entries and apply $\Mult_m$ again. This procedure is continued until there is only one entry left. The resulting network is called $\Mult_m^r$ and has $q(m+5)$ hidden layers and all parameters bounded by one. 

If $a,b,c,d \in [0,1],$ then by Lemma \ref{lem.mult} and triangle inequality, $|\Mult_m(a,b) - cd|\leq 2^{-m}+ |a-c| +|b-d|.$ By induction on the number of iterated multiplications $q,$ we therefore find that $|\Mult_m^r(\bx)- \prod_{i=1}^r x_i| \leq 3^{q-1} 2^{-m}\leq r^2 2^{-m}$ since $\log_2(3) \approx 1.58 < 2.$

From Lemma \ref{lem.mult} and the construction above, it follows that $\Mult_m^r(\bx)=0$ if one of the components of $\bx$ is zero.
\end{proof}

In the next step, we construct a sufficiently large network that approximates all monomials $x_1^{\alpha_1}\cdot \ldots \cdot x_r^{\alpha_r}$ for non-negative integers $\alpha_i$ up to a certain degree. As common, we use multi-index notation $\bx^{\balpha}:=x_1^{\alpha_1}\cdot \ldots \cdot x_r^{\alpha_r},$ where $\balpha= (\alpha_1, \ldots, \alpha_r)$ and $|\balpha |:= \sum_\ell |\alpha_\ell |$ is the degree of the monomial. 

The number of monomials with degree $|\balpha | < \gamma$ is denoted by $C_{r,\gamma}.$ Obviously, $C_{r,\gamma} \leq (\gamma +1)^r$ since each $\alpha_i$ has to take values in $\{0,1,\ldots,\lfloor \gamma \rfloor\}.$

\begin{lem2}
\label{lem.monomials}
For $\gamma >0$ and any positive integer $m,$ there exists a network 
\begin{align*}
	\Mon_{m,\gamma}^r \in \mF \big( 1+(m+5) \lceil \log_2 (\gamma \vee 1) \rceil , (r, 6\lceil \gamma\rceil C_{r,\gamma}, \ldots, 6\lceil \gamma\rceil C_{r,\gamma}, C_{r,\gamma}) \big),
\end{align*}
such that $\Mon_{m,\gamma}^r \in [0,1]^{C_{r,\gamma}}$ and
\begin{align*}
	\Big| \Mon_{m, \gamma}^r(\bx)  -  (\bx^{\balpha})_{|\balpha| < \gamma} \Big| _\infty \leq \gamma^2 2^{-m}, \quad \text{for all} \ \bx \in [0,1]^r.
\end{align*}
\end{lem2}

\begin{proof}
For $|\balpha|\leq 1,$ the monomials are linear or constant functions and there exists a shallow network in the class $\mF(1,(1,1,1))$ with output exactly $\bx^{\balpha}.$

By taking the multiplicity into account in \eqref{eq.1st_layer_mult}, Lemma \ref{lem.multi_mult} can be extended in a straightforward way to monomials. For $|\balpha|\geq 2,$ this shows the existence of a network in the class $\mF ( (m+5) \lceil \log_2 |\balpha| \rceil , (r, 6|\balpha|, \ldots, 6|\balpha|, 1))$ taking values in $[0,1]$ and approximating $\bx^{\balpha}$ up to sup-norm error $|\balpha|^2 2^{-m}.$ Using the parallelization and depth synchronization properties in Section 7.1 yields then the claim of the lemma. 
\end{proof}

\section{Proof of Theorem 5}

We follow the classical idea of function approximation by local Taylor approximations that has previously been used for network approximations in \cite{yarotski2017}. For a vector $\ba \in [0,1]^r$ define 
\begin{align}
	P_{\ba}^{\beta} f(\bx) = \sum_{0 \leq |\balpha|< \beta} (\partial^{\balpha} f) (\ba) \frac{(\bx - \ba)^{\balpha}}{\balpha !}.
	\label{eq.Taylor_approx}
\end{align}
By Taylor's theorem for multivariate functions, we have for a suitable $\xi \in [0,1],$ 
\begin{align*}
	f(\bx)= \sum_{\balpha : |\balpha| < \beta-1} (\partial^{\balpha} f )(\ba)\frac{(\bx-\ba)^{\balpha}}{\balpha !}
	+ 
	\sum_{\beta - 1 \leq |\balpha | < \beta } (\partial^{\balpha} f )(\ba+\xi (\bx-\ba))\frac{(\bx-\ba)^{\balpha}}{\balpha !}.
\end{align*}
We have $|(\bx - \ba)^{\balpha}| =\prod_i |x_i-a_i |^{\alpha_i} \leq |\bx-\ba|_\infty^{|\balpha|}.$ Consequently, for $f\in \mC_r^\beta([0,1]^r,K),$
\begin{align}
	&\big | f(\bx) - P_{\ba}^{\beta} f(\bx)  \big| \notag \\
	&\leq 
	\sum_{\beta-1 \leq |\balpha |< \beta} \frac{|(\bx-\ba)^{\balpha} |}{\balpha !}  
	 \big| (\partial^{\balpha} f )(\ba+\xi (\bx-\ba)) - (\partial^{\balpha} f )(\ba) \big| \label{eq.Hoeld_fct_approx}  \\
	&\leq 
	K |\bx-\ba|_\infty^{\beta}. \notag
\end{align}
We may also write \eqref{eq.Taylor_approx} as a linear combination of monomials 
\begin{align}
	P_{\ba}^{\beta} f(\bx) = \sum_{0 \leq |\bgamma| < \beta} \bx^{\bgamma} c_{\bgamma},
	\label{eq.Pf_monomial_sum}
\end{align}
for suitable coefficients $c_{\bgamma}.$ For convenience, we omit the dependency on $\ba$ in $c_{\bgamma}.$ Since $\partial^{\bgamma} P_{\ba}^{\beta} f(\bx) \, |_{\bx=0} = \bgamma ! c_{\bgamma},$ we must have
\begin{align*}
	c_{\bgamma} = \sum_{\bgamma \leq \balpha \&  |\balpha|< \beta} (\partial^{\balpha} f) (\ba) \frac{(- \ba)^{\balpha-\bgamma}}{\bgamma ! (\balpha -\bgamma)!}.
\end{align*} 
Notice that since $\ba \in [0,1]^r$ and $f\in \mC_r^\beta([0,1]^r,K),$ 
\begin{align}
	|c_{\bgamma}| \leq K/\bgamma ! \quad \text{and} \ \ \sum_{\bgamma \geq 0} |c_{\bgamma}| \leq K\prod_{j=1}^r \sum_{\gamma_j \geq 0} \frac{1}{\gamma_j!} = Ke^r.
	\label{eq.cgamma_bd}
\end{align}

Consider the set of grid points $\bD(M) :=\{\bx_{\bell} = (\ell_j/M)_{j=1,\ldots,r} : \bell =(\ell_1,\ldots, \ell_r) \in \{0,1,\ldots, M\}^r\}.$ The cardinality of this set is $(M+1)^r.$ We write $\bx_{\bell}=(x_j^{\bell})_{j}$ to denote the components of $\bx_{\bell}.$ Define
\begin{align*}
	P^\beta f (\bx) := \sum_{\bx_{\bell} \in \bD(M)} P_{\bx_{\bell}}^\beta f (\bx) \prod_{j=1}^r (1- M |x_j-x^{\bell}_j|)_+.
\end{align*}

\begin{lem2}
\label{lem.Hoeld_approx}
If $f\in \mC_r^\beta([0,1]^r, K),$ then $\| P^\beta f - f \|_{L^\infty[0,1]^r} \leq K M^{-\beta}.$
\end{lem2}

\begin{proof}
Since for all $\bx=(x_1, \ldots, x_r) \in [0,1]^r,$ 
\begin{align}
	\sum_{\bx_{\bell} \in \bD(M)} \prod_{j=1}^r (1- M |x_j-x^{\bell}_j|)_+
	= \prod_{j=1}^r \sum_{\ell=0}^{M}(1- M |x_j- \ell/M|)_+ =1,
	\label{eq.kernel_sum_one}
\end{align}
we have $f(\bx) = \sum_{\bx_{\bell} \in \bD(M): \|\bx- \bx_{\bell}\|_\infty \leq 1/M} \, f(\bx) \prod_{j=1}^r (1- M |x_j-x^{\bell}_j|)_+$ and with \eqref{eq.Hoeld_fct_approx},
\begin{align*}
	\big | P^\beta f (\bx) - f(\bx) \big| 
	\leq \max_{\bx_{\bell} \in \bD(M): \|\bx- \bx_{\bell}\|_\infty \leq 1/M} \big| P_{\bx_{\bell}}^\beta f (\bx) -  f(\bx) \big| 
	\leq K M^{-\beta}.
\end{align*}
\end{proof}

In a next step, we describe how to build a network that approximates $P^\beta f.$

\begin{lem2}
\label{lem.hat_fct_mult}
For any positive integers $M,m,$ there exists a network $$\Hat^r \in \mF\Big(2+(m+5)\lceil \log_2 r \rceil, \big(r, 6r(M+1)^r, \ldots, 6r(M+1)^r, (M+1)^r\big), s,1\Big)$$ with $s\leq 49 r^2 (M+1)^r (1+(m+5) \lceil \log_2 r \rceil),$ such that $\Hat^r \in [0,1]^{(M+1)^r}$ and for any $\bx=(x_1, \ldots, x_r) \in [0,1]^r,$
\begin{align*}
	\Big| \Hat^r(\bx) - \Big(\prod_{j=1}^r (1/M -  |x_j-x^{\bell}_j|)_+\Big)_{\bx_\ell \in \bD(M)} \Big|_\infty
	\leq r^2 2^{-m}.
\end{align*}
For any $\bx_\ell \in \bD(M),$ the support of the function $\bx \mapsto (\Hat^r(\bx))_{\bx_\ell}$ is moreover contained in the support of the function $\bx \mapsto \prod_{j=1}^r (1/M -  |x_j-x^{\bell}_j|)_+.$
\end{lem2}

\begin{proof}
The first hidden layer computes the functions $(x_j-\ell/M)_+$ and $(\ell/M-x_j)_+$ using $2r(M+1)$ units and $4r(M+1)$ non-zero parameters. The second hidden layer computes the functions $(1/M-  |x_j-\ell/M|)_+=(1/M -  (x_j-\ell/M)_+ - (\ell/M-x_j)_+)_+$ using $r(M+1)$ units and $3r(M+1)$ non-zero parameters. These functions take values in the interval $[0,1].$ This proves the result for $r=1.$ For $r>1,$ we compose the obtained network with networks that approximately compute the products $\prod_{j=1}^r (1/M - |x_j-\ell/M|)_+.$

By Lemma \ref{lem.multi_mult} there exist $\Mult_m^r$ networks in the class $$\mF((m+5) \lceil \log_2 r \rceil , (r, 6r,6r, \ldots, 6r,1))$$ computing $\prod_{j=1}^r (1/M - |x_j-x^{\bell}_j|)_+$ up to an error that is bounded by $r^2 2^{-m}.$ By (20), the number of non-zero parameters of one $\Mult_m^r$ network is bounded by $42 r^2 (1+(m+5) \lceil \log_2 r \rceil).$ As there are $(M+1)^r$ of these networks in parallel, this requires $6r (M+1)^r$ units in each hidden layer and $42 r^2 (M+1)^r (1+(m+5) \lceil \log_2 r \rceil)$ non-zero parameters for the multiplications. Together with the $7r(M+1)$ non-zero parameters from the first two layers, the total number of non-zero parameters is thus bounded by $49 r^2 (M+1)^r (1+(m+5) \lceil \log_2 r \rceil).$

By Lemma \ref{lem.multi_mult}, $\Mult_m^r(\bx)=0$ if one of the components of $\bx$ is zero. This shows that for any $\bx_\ell \in \bD(M),$ the support of the function $\bx \mapsto (\Hat^r(\bx))_{\bx_\ell}$ is contained in the support of the function $\bx \mapsto \prod_{j=1}^r (1/M -  |x_j-x^{\bell}_j|)_+.$
\end{proof}

\begin{proof}[Proof of Theorem 5]
All the constructed networks in this proof are of the form $\mF(L, \bp,s) =\mF(L, \bp,s, \infty)$ with $F=\infty.$ Let $M$ be the largest integer such that $(M+1)^r \leq N$ and define $L^*:=(m+5)\lceil \log_2 (\beta\vee r) \rceil .$ Thanks to \eqref{eq.cgamma_bd}, \eqref{eq.Pf_monomial_sum} and Lemma \ref{lem.monomials}, we can add one hidden layer to the network $\Mon_{m, \beta}^r$ to obtain a network $$Q_1 \in \mF\big( 2+L^* , (r, 6\lceil \beta\rceil C_{r,\beta}, \ldots, 6\lceil \beta\rceil C_{r,\beta}, C_{r,\beta}, (M+1)^r)\big),$$ such that $Q_1(\bx) \in [0,1]^{(M+1)^r}$ and for any $\bx\in [0,1]^r,$
\begin{align}
	\Big | Q_1(\bx) - \Big( \frac{P_{\bx_{\bell}}^\beta f (\bx)}{B}  + \frac 12\Big)_{\bx_{\bell} \in \bD(M)} \Big|_\infty \leq \beta^2 2^{-m}
	\label{eq.Q1_bd}
\end{align}
with $B:= \lceil 2Ke^r \rceil.$ By (20), the number of non-zero parameters in the $Q_1$ network is bounded by $6r(\beta+1) C_{r,\beta}+42(\beta+1)^2C_{r,\beta}^2(L^*+1) +C_{r,\beta}(M+1)^r.$

Recall that the network $\Hat^r$ computes the products of hat functions $\prod_{j=1}^r (1/M-  |x_j-x^{\bell}_j|)_+$ up to an error that is bounded by $r^2 2^{-m}.$ It requires at most $49 r^2 N (1+L^*)$ active parameters. Consider now the parallel network $(Q_1, \Hat^r).$ Observe that $C_{r,\beta} \leq (\beta+1)^r \leq N$ by the definition of $C_{r,\beta}$ and the assumptions on $N.$ By Lemma \ref{lem.hat_fct_mult}, the networks $Q_1$ and $\Hat^r$ can be embedded into a joint network $(Q_1, \Hat^r)$ with $2+L^*$ hidden layers, weight vector $(r, 6(r+\lceil \beta\rceil)N,\ldots, 6(r+\lceil \beta\rceil) N, 2(M+1)^r)$ and all parameters bounded by one. Using $C_{r,\beta} \vee (M+1)^r \leq N$ again, the number of non-zero parameters in the combined network $(Q_1, \Hat^r)$ is bounded by
\begin{align}
	\begin{split}
	&6r(\beta+1) C_{r,\beta}+42(\beta+1)^2C_{r,\beta}^2(L^*+1)\\
	&\ \ +C_{r,\beta}(M+1)^r + 49 r^2 N (1+L^*) \\
	&\leq 49 (r+\beta+1)^2 C_{r,\beta} N (1+L^*)\\
	&\leq 98 (r+\beta+1)^{3+r} N (m+5),
	\end{split}
	\label{eq.para_Q1_bd}
\end{align}
where for the last inequality, we used $C_{r,\beta} \leq (\beta+1)^r,$ the definition of $L^*$ and that for any $x\geq 1,$ $1+\lceil \log_2(x)\rceil\leq 2+\log_2(x)\leq 2(1+\log(x)) \leq 2x.$

Next, we pair the $\bx_{\bell}$-th entry of the output of $Q_1$ and $\Hat^r$ and apply to each of the $(M+1)^r$ pairs the $\Mult_m$ network described in Lemma \ref{lem.mult}. In the last layer, we add all entries. By Lemma \ref{lem.mult} this requires at most $42(m+5)(M+1)^r+(M+1)^r\leq 43 (m+5)N$ active parameters for the $(M+1)^r$ multiplications and the sum. Using Lemma \ref{lem.mult}, Lemma \ref{lem.hat_fct_mult}, \eqref{eq.Q1_bd} and triangle inequality, there exists a network $Q_2 \in \mF(3+(m+5)(1+\lceil \log_2 (\beta\vee r) \rceil), (r, 6(r+\lceil \beta\rceil) N, \ldots, 6(r+\lceil \beta\rceil)N,1))$ such that for any $\bx \in [0,1]^r,$ 
\begin{align}
\begin{split}
	&\Big | Q_2(\bx) - \sum_{\bx_{\bell} \in \bD(M)} \Big(\frac{P_{\bx_{\bell}}^\beta f (\bx)}{B} + \frac 12\Big)\prod_{j=1}^r \Big(\frac 1M-  |x_j-x^{\bell}_j|\Big)_+ \Big| \\
	&\leq \sum_{\bx_{\bell} \in \bD(M): \|\bx- \bx_{\bell}\|_\infty \leq 1/M} (1+r^2+\beta^2) 2^{-m}\\
	&\leq (1+r^2+\beta^2) 2^{r-m}.
	\end{split}\label{eq.Q2_estimate}
\end{align}
Here, the first inequality follows from the fact that the support of $(\Hat^r(\bx))_{\bx_\ell}$ is contained in the support of $\prod_{j=1}^r (1/M -  |x_j-x^{\bell}_j|)_+,$ see Lemma \ref{lem.hat_fct_mult}. Because of \eqref{eq.para_Q1_bd}, the network $Q_2$ has at most
\begin{align}
	\begin{split}
	&141 (r+\beta+1)^{3+r} N (m+5)
	\end{split}
	\label{eq.nr_param_Q2}
\end{align}
active parameters.

To obtain a network reconstruction of the function $f$, it remains to scale and shift the output entries. This is not entirely trivial because of the bounded parameter weights in the network. Recall that $B= \lceil 2K e^r \rceil .$ The network $x \mapsto BM^rx$ is in the class $\mF(3, (1,  M^r , 1, \lceil 2Ke^r\rceil,1))$ with shift vectors $\bv_j$ are all equal to zero and weight matrices $W_j$ having all entries equal to one. Because of $N \geq  (K+1)e^r,$ the number of parameters of this network is bounded by $2M^r+2\lceil 2Ke^r\rceil \leq 6 N.$ This shows existence of a network in the class $\mF(4, (1, 2, 2M^r ,2, 2\lceil 2Ke^r \rceil,  1)) $ computing $a \mapsto BM^r(a-c)$ with $c:=1/(2M^r).$ This network computes in the first hidden layer $(a-c)_+$ and $(c-a)_+$ and then applies the network $x \mapsto BM^rx$ to both units. In the output layer the second value is subtracted from the first one. This requires at most $6+12N$ active parameters. 

Because of \eqref{eq.Q2_estimate} and \eqref{eq.kernel_sum_one}, there exists a network $Q_3$ in
\begin{align*}
	 \mF\big(8+(m+5)(1+\lceil \log_2 (r\vee \beta) \rceil), (r, 6(r+\lceil \beta\rceil)N, \ldots, 6(r+\lceil \beta\rceil)N,1)\big)
\end{align*}
such that 
\begin{align*}
	&\Big | Q_3(\bx) - \sum_{\bx_{\bell} \in \bD(M)} P_{\bx_{\bell}}^\beta f (\bx)\prod_{j=1}^r \Big(1- M |x_j-x^{\bell}_j|\Big)_+ \Big| \\
	&\leq (2K+1)M^r (1+r^2+\beta^2) (2e)^r 2^{-m}, \ \  \text{for all} \ \bx \in [0,1]^r.
\end{align*}
With \eqref{eq.nr_param_Q2}, the number of non-zero parameters of $Q_3$ is bounded by 
\begin{align*}
	141 (r+\beta+1)^{3+r} N (m+6).
\end{align*}
Observe that  by construction $(M+1)^r \leq N \leq (M+2)^r \leq (3M)^r$ and hence $M^{-\beta}\leq N^{-\beta/r} 3^\beta.$ Together with Lemma \ref{lem.Hoeld_approx}, the result follows. 
\end{proof}

\section{Proofs for Section 7.2}

\begin{proof}[Proof of Lemma 4]
Throughout the proof we write $E=E_{f_0}.$ Define $\|g\|_n^2 := \frac{1}{n} \sum_{i=1}^n g(\bX_i)^2.$ For any estimator $\widetilde f$, we introduce $\widehat R_n(\widetilde f,f_0):=E\big[\| \widetilde f - f_0 \|_n^2 \big]$ for the empirical risk. In a first step, we show that we may restrict ourselves to the case $\log \mN_n \leq n.$ Since $R(\widehat f, f_0)\leq 4F^2,$ the upper bound trivially holds if $\log \mN_n \geq n.$ To see that also the lower bound is trivial in this case, let $\widetilde f \in \argmin_{f\in \mF} \sum_{i=1}^n (Y_i - f(\bX_i))^2$ be a (global) empirical risk minimizer. Observe that
\begin{align}
	\widehat R_n(\widehat f, f_0) - \widehat R_n(\widetilde f, f_0)
	= \Delta_n + E\Big[\frac{2}{n} \sum_{i=1}^n \eps_i \widehat f(\bX_i)\Big] - E\Big[\frac{2}{n} \sum_{i=1}^n \eps_i \widetilde f(\bX_i)\Big].
	\label{eq.lb_el_expansion}
\end{align} 
From this equation, it follows that $\Delta_n \leq 8F^2$ and this implies the lower bound in the statement of the lemma for $\log \mN_n \geq n.$ We may therefore assume $\log \mN_n \leq n.$ The proof is divided into four parts which are denoted by $(I)-(IV).$ 

{\itshape (I):} We relate the risk $R( \widehat f, f_0)= E[(\widehat f(\bX) - f_0(\bX) )^2 ]$ to its empirical counterpart $\widehat R_n( \widehat f, f_0)$ via the inequalities
\begin{align*}
		&(1- \eps) \widehat R_n(\widehat f,f_0)
		- \frac{F^2}{n\eps} \big(15\log \mN_n + 75\big)-26\delta F \\
		&\leq 
		R( \widehat f, f_0)
		\leq (1+ \eps) \Big( \widehat R_n(\widehat f,f_0)
		+(1+\eps) \frac{F^2}{n\eps} \big(12\log \mN_n + 70\big) + 26\delta F \Big).
\end{align*}
{\itshape (II):} For any estimator $\widetilde f$ taking values in $\mF,$ 
\begin{align*}
	\Big| E\Big[ \frac 2n \sum_{i=1}^n \eps_i  \widetilde f(\bX_i) \Big] \Big|
	\leq 2\sqrt{\frac{\widehat R_n(\widetilde f,f_0) (3\log \mN_n +1)}n} + 6\delta.
\end{align*}
{\itshape (III):} We have
\begin{align*}
	\widehat R_n(\widehat f,f_0)
	\leq (1+\eps)\Big[\inf_{f\in \mF} E\big[\big(f(\bX) - f_0(\bX) \big)^2 \big] + 6\delta + F^2\frac{6\log \mN_n+2}{n\eps}+\Delta_n \Big].
\end{align*}
{\itshape (IV):} We have
\begin{align*}
	\widehat R_n(\widehat f,f_0)
	\geq (1-\eps) \Big( \Delta_n - \frac{3\log \mN_n +1}{n\eps} - 12\delta\Big).
\end{align*}
Combining $(I)$ and $(IV)$ gives the lower bound of the assertion since $F \geq 1.$ The upper bound follows from $(I)$ and $(III).$

{\itshape (I):} Given a minimal $\delta$-covering of $\mF,$ denote the centers of the balls by $f_j.$ By construction there exists a (random) $j^*$ such that $\|\widehat f-f_{j^*}\|_\infty \leq \delta.$ Without loss of generality, we can assume that $\|f_j\|_\infty \leq F.$ Generate i.i.d. random variables $\bX_i',$ $i=1, \ldots, n$ with the same distribution as $\bX$ and independent of $(\bX_i)_{i=1, \ldots,n}.$ Using that $\|f_j\|_\infty, \|f_0\|_\infty, \delta  \leq F,$
\begin{align*}
	&\big | R( \widehat f, f_0) - \widehat R_n( \widehat f, f_0)\big| \\
	&= \Big| E\Big[\frac{1}{n} \sum_{i=1}^n \big( \widehat f(\bX_i') -f_0(\bX_i') \big)^2 
	- \frac{1}{n} \sum_{i=1}^n \big( \widehat f(\bX_i) -f_0(\bX_i) \big)^2 \Big] \Big| \notag \\
	&\leq E\Big[ \Big|\frac 1n \sum_{i=1}^n g_{j^*}(\bX_i, \bX_i') \Big| \Big] 
	+ 9\delta F,
\end{align*} 
with $g_{j^*}(\bX_i, \bX_i') := (f_{j^*}(\bX_i') -f_0(\bX_i') \big)^2- ( f_{j^*}(\bX_i) -f_0(\bX_i) )^2.$ Define $g_j$ in the same way with $f_{j^*}$ replaced by $f_j.$ Similarly, set $r_j := \sqrt{n^{-1} \log \mN_n } \vee E^{1/2}[(f_j(\bX) -f_0(\bX) )^2]$ and define $r^*$ as $r_j$ for $j=j^*,$ which is the same as 
\begin{align*}
	r^*&=\sqrt{n^{-1} \log \mN_n } \vee E^{1/2}[(f_{j^*}(\bX) -f_0(\bX) )^2| (\bX_i,Y_i)_i] 	\\
	&\leq \sqrt{n^{-1} \log \mN_n } + E^{1/2}[(\widehat f(\bX) -f_0(\bX) )^2| (\bX_i,Y_i)_i] + \delta,
\end{align*}
where the last part follows from triangle inequality and $\|f_{j^*}-\widehat f\|_\infty \leq \delta.$

For random variables $U,T,$ Cauchy-Schwarz gives $E[UT]\leq E^{1/2}[U^2]E^{1/2}[T^2].$ Choose $U= E^{1/2}[(\widehat f(\bX) -f_0(\bX) )^2| (\bX_i,Y_i)_i] $ and $T:=\max_j  | \sum_{i=1}^n g_{j}(\bX_i, \bX_i') /(r_j F)|.$ Using that $E[U^2]=R( \widehat f, f_0)$
\begin{align}
\begin{split}
	&\big | R( \widehat f, f_0) - \widehat R_n( \widehat f, f_0) \big| \\
	&\leq   \frac{F}n R( \widehat f, f_0)^{1/2} E^{1/2}[T^2]+\frac{F}{n}\Big(\sqrt{\frac{\log \mN_n }n}+ \delta\Big)E[T]  + 9\delta F.
\end{split}	
	\label{eq.Rf_estimate}
\end{align}
Observe that $E[g_j(\bX_i, \bX_i') ]=0,$ $|g_j(\bX_i, \bX_i')|\leq 4F^2$ and 
\begin{align*}
	\Var\big(g_{j}(\bX_i, \bX_i')\big)
	&= 2\Var\big((f_j(\bX_i) -f_0(\bX_i) )^2\big) \\
	&\leq 2E\big[\big(f_j(\bX_i) -f_0(\bX_i) \big)^4\big] \\
	&\leq 8F^2 r_j^2.
\end{align*}
Bernstein's inequality states that for independent and centered random variables $U_1, \ldots, U_n,$ satisfying $|U_i|\leq M,$ $P(|\sum_{i=1}^n U_i | \geq t)\leq 2\exp(-  t^2/[2Mt/3 + 2\sum_{i=1}^n\Var( U_i)]),$ cf. \cite{vdVaartWellner}. Combining Bernstein's inequality with a union bound argument yields
\begin{align*}
	P( T \geq t ) 
	\leq 1\wedge  2 \mN_n \max_j \, \exp\Big( - \frac{t^2}{8 t/(3r_j) +16 n}\Big).
\end{align*}
The first term in the denominator of the exponent dominates for large $t.$ Since $r_j\geq  \sqrt{n^{-1} \log \mN_n},$ we have $P( T \geq t ) \leq 2 \mN_n \exp( - 3t \sqrt{\log \mN_n}/(16\sqrt{n}))$ for all $t \geq 6 \sqrt{n \log \mN_n}.$ We therefore find
\begin{align*}
	E[T] & =  \int_0^\infty P(T \geq t) \, dt \\
	&\leq 6 \sqrt{n \log \mN_n} + \int_{6 \sqrt{n \log \mN_n}}^\infty 2\mN_n \exp\Big( - \frac{3t \sqrt{\log \mN_n}}{16\sqrt{n}}\Big) \, dt \\
	&\leq 6 \sqrt{n \log \mN_n}  + \frac{32}{3} \sqrt{\frac{n}{\log \mN_n}}.
\end{align*}
By assumption $\mN_n \geq 3$ and hence $\log \mN_n \geq 1.$ We can argue in a similar way as for the upper bound of $E[T]$ in order to find for the second moment
\begin{align*}
	E[T^2] 
	&=  \int_0^\infty P(T^2\geq u) \, du=\int_0^\infty P(T\geq \sqrt{u}) \, du \\
	&\leq 36 n\log \mN_n + \int_{36n \log \mN_n}^\infty 2\mN_n \exp\Big( - \frac{3\sqrt{u} \sqrt{\log \mN_n}}{16\sqrt{n}}\Big) \, du \\
	& \leq 36 n\log \mN_n + 2^8 n,
\end{align*}
where the second inequality uses $\int_{b^2}^\infty e^{-\sqrt{u} a} du = 2 \int_b^\infty se^{-sa} ds = 2(ba+1)e^{-ba}/a^2$ which can be obtained from substitution and integration by parts. With \eqref{eq.Rf_estimate} and $1\leq \log \mN_n \leq n,$
\begin{align}
\begin{split}
		&\big | R( \widehat f, f_0) - \widehat R_n( \widehat f, f_0) \big| \\
		&\leq  
		\frac{F}{n} R( \widehat f, f_0)^{1/2}\big(36 n\log \mN_n + 2^8 n\big)^{1/2}
		+F\frac{6\log \mN_n+11}{n} + 26 \delta F
\end{split}		
	\label{eq.exp_bd_in_proof1}
\end{align}
Let $a,b,c,d$ be positive real numbers, such that $|a-b|\leq 2\sqrt{a}c + d.$ Then, for any $\eps\in (0,1],$ 
\begin{align}
	(1-\eps)b -d - \frac{c^2}{\eps}  \leq a \leq (1+\eps) (b+d) + \frac{(1+\eps)^2}{\eps} c^2.
	\label{eq.abc_identity}
\end{align}
To see this observe that $|a-b|\leq 2\sqrt{a}c + d$ implies $a \leq \eps a/(1+\eps) +(1+\eps)c^2/\eps + (b+d).$ Rearranging the terms yields the upper bound. For the lower bound, we use the same argument and find $a\geq -\eps a/(1-\eps) - (1-\eps) c^2/\eps + (b-d)$ which gives \eqref{eq.abc_identity}. With $a= R(\widehat f , f_0),$ $b= \widehat R_n( \widehat f, f_0),$ $c=F\big(9n\log \mN_n + 64 n\big)^{1/2}/n,$ $d= F(6\log \mN_n+11)/n + 26 \delta F$ the asserted inequality of $(I)$ follows from \eqref{eq.exp_bd_in_proof1}. Notice that we have used $2\leq (1+\eps)/\eps$ for the upper bound. 

{\itshape (II):} Given an estimator $\widetilde f$ taking values in $\mF,$ let $j'$ be such that $\|\widetilde f-f_{j'}\|_\infty \leq \delta.$ We have $| E[\sum_{i=1}^n \eps_i ( \widetilde f(\bX_i) - f_{j'}(\bX_i))]|\leq  \delta E[\sum_{i=1}^n |\eps_i|] \leq n \delta.$ Since $E[\eps_i f_0(\bX_i)] = E[E[\eps_i f_0(\bX_i) \, | \, \bX_i]]=0,$ we also find
\begin{align}
	\begin{split}
	\Big| E\Big[ \frac 2n \sum_{i=1}^n \eps_i  \widetilde f(\bX_i) \Big] \Big|
	&=
	\Big| E\Big[ \frac 2n \sum_{i=1}^n \eps_i  \big(\widetilde f(\bX_i) - f_0(\bX_i) \big)\Big] \Big| \\
	&\leq 2\delta + \frac 2{\sqrt{n}} E\Big[ (\|\widetilde f - f_0\|_n + \delta ) |\xi_{j'} | \Big]
	\end{split}
	\label{eq.oracle_gen4}
\end{align}
with
\begin{align*}
	\xi_j := \frac{\sum_{i=1}^n \eps_i (f_j(\bX_i)-f_0(\bX_i))}{\sqrt{n} \|f_j - f_0\|_n } .
\end{align*}
Conditionally on $(\bX_i)_i,$ $\xi_j \sim \mN(0,1).$  With Lemma \ref{lem.max_normal}, we obtain $E[\xi_{j'}^2] \leq E[ \max_j \xi_j^2] \leq 3 \log \mN_n + 1.$ Using Cauchy-Schwarz,
\begin{align}
	E\Big[ (\|\widetilde f - f_0\|_n + \delta ) |\xi_{j'} | \Big]
	\leq \Big( \widehat R_n(\widetilde f,f_0)^{1/2} + \delta \Big) \sqrt{3 \log \mN_n + 1}.
	 \label{eq.oracle_gen5}
\end{align}
Because of $\log \mN_n \leq n,$ we have $2n^{-1/2}\delta \sqrt{3\log \mN_n +1} \leq 4\delta.$ Together with \eqref{eq.oracle_gen4} and \eqref{eq.oracle_gen5} the result follows.

{\itshape (III):} For any fixed $f\in\mF,$ $E[\tfrac 1n \sum_{i=1}^n (Y_i-\widehat f(\bX_i))^2] \leq E[\tfrac 1n\sum_{i=1}^n (Y_i-f(\bX_i))^2]+\Delta_n.$ Because of $\bX_i \stackrel{\mD}{=}\bX$ and $f$ being deterministic, we have $E[\| f - f_0\|_n^2]= E[( f(\bX) - f_0(\bX) )^2].$ Since also $E[\eps_i f(\bX_i)] = E[E[\eps_i f(\bX_i) \, | \, \bX_i]]=0,$
\begin{align*}
	\widehat R_n (\widehat f, f_0)
	&\leq E[\| f - f_0\|_n^2]  + E\Big[\frac{2}{n}\sum_{i=1}^n \eps_i \widehat f(\bX_i)\big) \Big]+\Delta_n \\
	&\leq E[( f(\bX) - f_0(\bX) )^2] + 2\sqrt{\frac{\widehat R_n(\widehat f,f_0) (3\log \mN_n +1)}n} + 6\delta + \Delta_n
\end{align*}
using for the second inequality $(II).$ Applying \eqref{eq.abc_identity} to $a:=\widehat R_n(\widehat f,f_0),$ $b:=0,$ $c:= \sqrt{(3 \log \mN_n + 1)/n},$ $d:= E[(f(\bX) - f_0(\bX))^2]+6\delta+\Delta_n,$ yields $(III).$

{\itshape (IV):} Let $\widetilde f \in \argmin_{f\in \mF} \sum_{i=1}^n (Y_i - f(\bX_i))^2$ be an empirical risk minimizer. Using \eqref{eq.lb_el_expansion}, $(II)$ and $(1-\eps)/\eps+1=1/\eps,$ we find
\begin{align*}
	&\widehat R_n(\widehat f, f_0) - \widehat R_n(\widetilde f, f_0)\\
	&\geq  
	\Delta_n  - 2\sqrt{\frac{\widehat R_n(\widehat f,f_0) (3\log \mN_n +1)}n} - 2\sqrt{\frac{\widehat R_n(\widetilde f,f_0) (3\log \mN_n +1)}n} - 12\delta \\
	&\geq 
	\Delta_n - \frac{\eps }{1-\eps} \widehat R_n(\widehat f,f_0) 
	- \widehat R_n(\widetilde f,f_0) - \frac{3\log \mN_n +1}{n\eps} - 12\delta.
\end{align*}	
Rearranging of the terms leads then to the conclusion of $(IV).$ 
\end{proof}

\begin{lem2}
\label{lem.max_normal}
Let $\eta_j \sim \mathcal{N}(0,1),$ then $E[\max_{j=1, \ldots,M} \eta_j^2] \leq 3\log M +1.$
\end{lem2}


\begin{proof}
Let $Z=\max_{j=1,\ldots,M}\eta_j^2.$ Since $Z\leq \sum_j \eta_j^2,$ we have $E[Z]\leq M.$ For $M\in \{1,2,3\}$ it can be checked that $M \leq 3\log(M)+1.$ It is therefore enough to consider $M \geq 4.$ Mill's ratio gives $P(|\eta_1|\geq \sqrt{t})=2P(\eta_1\geq \sqrt{t})\leq 2e^{-t/2}/(\sqrt{2\pi t}).$ For any $T,$ we have using the union bound,
\begin{align*}
	E[Z]
	&= \int_0^\infty P(Z\geq t) dt
	\leq T + \int_T^\infty P(Z\geq t) dt
	\leq T+M \int_T^\infty P(\eta_1^2 \geq t) dt \\
	& \leq T+ M \int_T^\infty \frac{2}{\sqrt{2\pi t}} e^{-t/2} dt
	\leq T+ \frac{2M}{\sqrt{2\pi T}}  \int_T^\infty e^{-t/2} dt \\
	&=  T+ \frac{4M}{\sqrt{2\pi T}}e^{-T/2}.
\end{align*}
For $T=2\log M$ and $M\geq 4,$ we find $E[Z]\leq 2\log M+ 2/\sqrt{\pi \log M}\leq 2\log M+1.$
\end{proof}

\begin{proof}[Proof of Lemma 5]
Given a neural network $$f(\bx) = W_L \sigma_{\bv_L}  W_{L-1}  \sigma_{\bv_{L-1}}  \cdots  W_1 \sigma_{\bv_1}  W_0\bx,$$ define for $k \in \{1, \ldots, L\},$ $A_k^+ f : [0,1]^d \rightarrow \mathbb{R}^{p_k},$
\begin{align*}
	&A_k^+ f(\bx)  = \sigma_{\bv_k}   W_{k-1}  \sigma_{\bv_{k-1}}  \cdots  W_1 \sigma_{\bv_1}  W_0\bx
\end{align*}
and $A_k^- f : \mathbb{R}^{p_{k-1}} \rightarrow \mathbb{R}^{p_{L+1}},$
\begin{align*}
	&A_k^- f(\by)  = W_L  \sigma_{\bv_L}   W_{L-1}  \sigma_{\bv_{L-1}}  \cdots  W_k \sigma_{\bv_k}  W_{k-1}\by.
\end{align*}
Set $A_0^+ f(\bx) = A_{L+2}^- f(\bx) =\bx$ and notice that for $f \in \mF(L, \bp),$ $| A_k^+ f(\bx) |_\infty \leq \prod_{\ell =0}^{k-1} (p_\ell +1).$ For a multivariate function $h,$ we say that $h$ is Lipschitz if $|h(\bx)-h(\by)|_\infty \leq L|\bx-\by|_\infty$ for all $\bx,\by$ in the domain. The smallest $L$ is the Lipschitz constant. Composition of two Lipschitz functions with Lipschitz constants $L_1$ and $L_2$ gives again a Lipschitz function with Lipschitz constant $L_1L_2.$ Therefore, the Lipschitz constant of $A_k^- f$ is bounded by $\prod_{\ell = k-1}^L p_\ell.$ Fix $\eps>0.$ Let $f, f^* \in \mF(L, \bp, s)$ be two network functions, such that all parameters are at most $\eps$ away from each other. Denote the parameters of $f$ by $(v_k, W_k)_k$ and the parameters of $f^*$ by $(v_k^*,W_k^*)_k.$ Then, we can bound the absolute value of the difference by ($\sigma_{\bv_{L+1}}$ is defined as the identity)
\begin{align*}
	\big |f(\bx) - f^*(\bx)\big |
	&\leq   \sum_{k=1}^{L+1} \Big | A_{k+1}^-f  \sigma_{\bv_k}W_{k-1}  A_{k-1}^+ f^*(\bx)
	- A_{k+1}^-f  \sigma_{\bv_k^*}W_{k-1}^*  A_{k-1}^+ f^*(\bx) \Big|\\
	&\leq  \sum_{k=1}^{L+1} \Big(\prod_{\ell = k}^L p_\ell \Big)  \big |  \sigma_{\bv_k}W_{k-1}  A_{k-1}^+ f^*(\bx)
	-  \sigma_{\bv_k^*}W_{k-1}^*  A_{k-1}^+ f^*(\bx) \big |_\infty \\
	&\leq  \sum_{k=1}^{L+1} \Big( \prod_{\ell = k}^L p_\ell \Big) \Big(\big |  (W_{k-1} -W_{k-1}^*)  A_{k-1}^+ f^*(\bx) \big |_\infty + |\bv_k-\bv_k^*|_\infty \Big)\\
	&\leq \eps  \sum_{k=1}^{L+1} \Big( \prod_{\ell = k}^L p_\ell \Big)  \Big( p_{k-1} \big | A_{k-1}^+ f^*(\bx) \big |_\infty+ 1 \Big)\\
	&\leq \eps V (L+1),
\end{align*}
using $V:=\prod_{\ell = 0}^{L+1} (p_\ell +1)$ for the last step. By (20) the total number of parameters is therefore bounded by $T:=\sum_{\ell=0}^L (p_\ell +1) p_{\ell+1}\leq (L+1)2^{-L}\prod_{\ell = 0}^{L+1} (p_\ell +1)\leq V$ and there are $\binom{T}{s}\leq V^s$ combinations to pick $s$ non-zero parameters. Since all the parameters are bounded in absolute value by one, we can discretize the non-zero parameters with grid size $\delta / (2(L+1) V)$ and obtain for the covering number
\begin{align*}
	\mathcal{N}\Big( \delta, \mF(L, \bp, s, \infty), \|\cdot \|_\infty \Big)
	&\leq \sum_{s^* \leq s} \big( 2\delta^{-1} (L+1) V^2 \big)^{s^*}
	\leq \big( 2\delta^{-1} (L+1) V^2 \big)^{s+1}.
\end{align*}
Taking logarithms yields the result.
\end{proof}

\bibliographystyle{acm}       
\bibliography{bibDL}           

\begin{thebibliography}{10}

\bibitem{2019arXiv190311257A}
{\sc {Ahmad}, S., and {Scheinkman}, L.}
\newblock {How can we be so dense? The benefits of using highly sparse
  representations}.
\newblock {\em arXiv e-prints\/} (2019), arXiv:1903.11257.

\bibitem{anthony1999}
{\sc Anthony, M., and Bartlett, P.~L.}
\newblock {\em Neural network learning: theoretical foundations}.
\newblock Cambridge University Press, Cambridge, 1999.

\bibitem{Bach2017}
{\sc Bach, F.}
\newblock Breaking the curse of dimensionality with convex neural networks.
\newblock {\em Journal of Machine Learning Research 18}, 19 (2017), 1--53.

\bibitem{baraud2014}
{\sc Baraud, Y., and Birg\'e, L.}
\newblock Estimating composite functions by model selection.
\newblock {\em Ann. Inst. Henri Poincar\'e Probab. Stat. 50}, 1 (2014),
  285--314.

\bibitem{barron1993}
{\sc Barron, A.~R.}
\newblock Universal approximation bounds for superpositions of a sigmoidal
  function.
\newblock {\em IEEE Transactions on Information Theory 39}, 3 (1993), 930--945.

\bibitem{barron1994}
{\sc Barron, A.~R.}
\newblock Approximation and estimation bounds for artificial neural networks.
\newblock {\em Machine Learning 14}, 1 (1994), 115--133.

\bibitem{2018arXiv180903090B}
{\sc {Barron}, A.~R., and {Klusowski}, J.~M.}
\newblock Approximation and estimation for high-dimensional deep learning
  networks.
\newblock {\em arXiv e-prints\/} (2018), arXiv:1809.03090.

\bibitem{bauer2017}
{\sc Bauer, B., and Kohler, M.}
\newblock On deep learning as a remedy for the curse of dimensionality in
  nonparametric regression.
\newblock {\em Ann. Statist. 47}, 4 (2019), 2261--2285.

\bibitem{BelkinRT2019}
{\sc Belkin, M., Rakhlin, A., and Tsybakov, A.~B.}
\newblock Does data interpolation contradict statistical optimality?
\newblock In {\em Proceedings of Machine Learning Research\/} (2019),
  K.~Chaudhuri and M.~Sugiyama, Eds., vol.~89 of {\em Proceedings of Machine
  Learning Research}, PMLR, pp.~1611--1619.

\bibitem{boelcskei2017}
{\sc {B{\"o}lcskei}, H., {Grohs}, P., {Kutyniok}, G., and {Petersen}, P.}
\newblock {Optimal approximation with sparsely connected deep neural networks}.
\newblock {\em SIAM Journal on Mathematics of Data Science 1}, 1 (2019), 8--45.

\bibitem{candes2002}
{\sc {Cand\`es, Emmanuel J.}}
\newblock New ties between computational harmonic analysis and approximation
  theory.
\newblock In {\em Approximation theory, {X} ({S}t. {L}ouis, {MO}, 2001)},
  Innov. Appl. Math. Vanderbilt Univ. Press, Nashville, TN, 2002, pp.~87--153.

\bibitem{choromanska2015}
{\sc Choromanska, A., Henaff, M., Mathieu, M., Arous, G.~B., and LeCun, Y.}
\newblock The loss surface of multilayer networks.
\newblock In {\em Aistats\/} (2015), pp.~192--204.

\bibitem{cohen1993}
{\sc Cohen, A., Daubechies, I., and Vial, P.}
\newblock Wavelets on the interval and fast wavelet transforms.
\newblock {\em Applied and Computational Harmonic Analysis 1}, 1 (1993), 54 --
  81.

\bibitem{Cybenko1989}
{\sc Cybenko, G.}
\newblock Approximation by superpositions of a sigmoidal function.
\newblock {\em Mathematics of Control, Signals and Systems 2}, 4 (1989),
  303--314.

\bibitem{ECKLE2019232}
{\sc Eckle, K., and Schmidt-Hieber, J.}
\newblock {A comparison of deep networks with ReLU activation function and
  linear spline-type methods}.
\newblock {\em Neural Networks 110\/} (2019), 232 -- 242.

\bibitem{frankle2018lottery}
{\sc Frankle, J., and Carbin, M.}
\newblock The lottery ticket hypothesis: Finding sparse, trainable neural
  networks, 2018.
\newblock ArXiv:1803.03635.

\bibitem{GaierHa2019}
{\sc {Gaier}, A., and {Ha}, D.}
\newblock Weight agnostic neural networks.
\newblock {\em arXiv e-prints\/} (2019), arXiv:1906.04358.

\bibitem{2019arXiv190209574G}
{\sc {Gale}, T., {Elsen}, E., and {Hooker}, S.}
\newblock The state of sparsity in deep neural networks.
\newblock {\em arXiv e-prints\/} (2019), arXiv:1902.09574.

\bibitem{gine2006}
{\sc Gin\'e, E., and Koltchinskii, V.}
\newblock Concentration inequalities and asymptotic results for ratio type
  empirical processes.
\newblock {\em Ann. Probab. 34}, 3 (2006), 1143--1216.

\bibitem{glorot2011}
{\sc Glorot, X., Bordes, A., and Bengio, Y.}
\newblock Deep sparse rectifier neural networks.
\newblock In {\em Aistats\/} (2011), vol.~15, pp.~315--323.

\bibitem{Goodfellow-et-al-2016-Book}
{\sc Goodfellow, I., Bengio, Y., and Courville, A.}
\newblock {\em Deep Learning}.
\newblock MIT Press, 2016.

\bibitem{MR3432148}
{\sc Gower, R.~M., and Richt\'{a}rik, P.}
\newblock Randomized iterative methods for linear systems.
\newblock {\em SIAM J. Matrix Anal. Appl. 36}, 4 (2015), 1660--1690.

\bibitem{MR1270012}
{\sc Green, P.~J., and Silverman, B.~W.}
\newblock {\em Nonparametric regression and generalized linear models}, vol.~58
  of {\em Monographs on Statistics and Applied Probability}.
\newblock Chapman \& Hall, London, 1994.

\bibitem{2018arXiv180208246G}
{\sc {Gunasekar}, S., {Lee}, J., {Soudry}, D., and {Srebro}, N.}
\newblock Characterizing implicit bias in terms of optimization geometry.
\newblock {\em arXiv e-prints\/} (2018), arXiv:1802.08246.

\bibitem{gyorfi2002}
{\sc Gy\"orfi, L., Kohler, M., Krzy\.zak, A., and Walk, H.}
\newblock {\em A distribution-free theory of nonparametric regression}.
\newblock Springer Series in Statistics. Springer-Verlag, New York, 2002.

\bibitem{hamers2006}
{\sc Hamers, M., and Kohler, M.}
\newblock Nonasymptotic bounds on the {$L^2$-}error of neural network
  regression estimates.
\newblock {\em Annals of the Institute of Statistical Mathematics 58}, 1
  (2006), 131--151.

\bibitem{NIPS2015_5784}
{\sc Han, S., Pool, J., Tran, J., and Dally, W.}
\newblock Learning both weights and connections for efficient neural network.
\newblock In {\em Advances in Neural Information Processing Systems 28},
  C.~Cortes, N.~D. Lawrence, D.~D. Lee, M.~Sugiyama, and R.~Garnett, Eds.
  Curran Associates, Inc., 2015, pp.~1135--1143.

\bibitem{NIPS1992_647}
{\sc Hassibi, B., and Stork, D.~G.}
\newblock Second order derivatives for network pruning: Optimal brain surgeon.
\newblock In {\em Advances in Neural Information Processing Systems 5}, S.~J.
  Hanson, J.~D. Cowan, and C.~L. Giles, Eds. Morgan-Kaufmann, 1993,
  pp.~164--171.

\bibitem{he2015}
{\sc He, K., and Sun, J.}
\newblock Convolutional neural networks at constrained time cost.
\newblock In {\em CVPR\/} (2015), pp.~5353--5360.

\bibitem{he2016}
{\sc He, K., Zhang, X., Ren, S., and Sun, J.}
\newblock Deep residual learning for image recognition.
\newblock In {\em CVPR\/} (2016), pp.~770--778.

\bibitem{hendrycks2018benchmarking}
{\sc Hendrycks, D., and Dietterich, T.}
\newblock Benchmarking neural network robustness to common corruptions and
  perturbations.
\newblock In {\em International Conference on Learning Representations\/}
  (2019).

\bibitem{Hornik1989}
{\sc Hornik, K., Stinchcombe, M., and White, H.}
\newblock Multilayer feedforward networks are universal approximators.
\newblock {\em Neural Networks 2}, 5 (1989), 359 -- 366.

\bibitem{Hornik1990}
{\sc Hornik, K., Stinchcombe, M., and White, H.}
\newblock Universal approximation of an unknown mapping and its derivatives
  using multilayer feedforward networks.
\newblock {\em Neural Networks 3}, 5 (1990), 551 -- 560.

\bibitem{horowitz2007}
{\sc Horowitz, J.~L., and Mammen, E.}
\newblock Rate-optimal estimation for a general class of nonparametric
  regression models with unknown link functions.
\newblock {\em Ann. Statist. 35}, 6 (2007), 2589--2619.

\bibitem{juditsky2009}
{\sc Juditsky, A.~B., Lepski, O.~V., and Tsybakov, A.~B.}
\newblock Nonparametric estimation of composite functions.
\newblock {\em Ann. Statist. 37}, 3 (2009), 1360--1404.

\bibitem{MR1895892}
{\sc Kerkyacharian, G., and Picard, D.}
\newblock Minimax or maxisets?
\newblock {\em Bernoulli 8}, 2 (2002), 219--253.

\bibitem{klusowski2016a}
{\sc {Klusowski}, J.~M., and {Barron}, A.~R.}
\newblock Risk bounds for high-dimensional ridge function combinations
  including neural networks.
\newblock {\em ArXiv e-prints\/} (2016).

\bibitem{klusowski2016b}
{\sc {Klusowski}, J.~M., and {Barron}, A.~R.}
\newblock Uniform approximation by neural networks activated by first and
  second order ridge splines.
\newblock {\em ArXiv e-prints\/} (2016).

\bibitem{kohler2005}
{\sc Kohler, M., and Krzyzak, A.}
\newblock Adaptive regression estimation with multilayer feedforward neural
  networks.
\newblock {\em Journal of Nonparametric Statistics 17}, 8 (2005), 891--913.

\bibitem{kohler2017}
{\sc Kohler, M., and Krzy\.zak, A.}
\newblock Nonparametric regression based on hierarchical interaction models.
\newblock {\em IEEE Trans. Inform. Theory 63}, 3 (2017), 1620--1630.

\bibitem{koltchinskii2006}
{\sc Koltchinskii, V.}
\newblock Local {R}ademacher complexities and oracle inequalities in risk
  minimization.
\newblock {\em Ann. Statist. 34}, 6 (2006), 2593--2656.

\bibitem{krizhevsky2012}
{\sc Krizhevsky, A., Sutskever, I., and Hinton, G.~E.}
\newblock Imagenet classification with deep convolutional neural networks.
\newblock In {\em Advances in Neural Information Processing Systems 25},
  F.~Pereira, C.~J.~C. Burges, L.~Bottou, and K.~Q. Weinberger, Eds. Curran
  Associates, Inc., 2012, pp.~1097--1105.

\bibitem{NIPS1989_250}
{\sc LeCun, Y., Denker, J.~S., and Solla, S.~A.}
\newblock Optimal brain damage.
\newblock In {\em Advances in Neural Information Processing Systems 2}, D.~S.
  Touretzky, Ed. Morgan-Kaufmann, 1990, pp.~598--605.

\bibitem{Leshno1993}
{\sc Leshno, M., Lin, V.~Y., Pinkus, A., and Schocken, S.}
\newblock Multilayer feedforward networks with a nonpolynomial activation
  function can approximate any function.
\newblock {\em Neural Networks 6}, 6 (1993), 861 -- 867.

\bibitem{liang2016}
{\sc {Liang}, S., and {Srikant}, R.}
\newblock {Why deep neural networks for function approximation?}
\newblock In {\em ICLR 2017\/} (2017).

\bibitem{liao2017}
{\sc {Liao}, Q., and {Poggio}, T.}
\newblock {Theory II: Landscape of the empirical risk in deep learning}.
\newblock {\em ArXiv e-prints\/} (2017).

\bibitem{gower2015stochastic}
{\sc {Mansel Gower}, R., and {Richtarik}, P.}
\newblock Stochastic dual ascent for solving linear systems.
\newblock {\em arXiv e-prints\/} (2015), arXiv:1512.06890.

\bibitem{massart2007}
{\sc Massart, P., and Picard, J.}
\newblock {\em Concentration Inequalities and Model Selection}.
\newblock Lecture notes in mathematics. Springer, 2007.

\bibitem{mccaffrey1994}
{\sc McCaffrey, D.~F., and Gallant, A.~R.}
\newblock Convergence rates for single hidden layer feedforward networks.
\newblock {\em Neural Networks 7}, 1 (1994), 147 -- 158.

\bibitem{mhaskar1993}
{\sc Mhaskar, H.~N.}
\newblock Approximation properties of a multilayered feedforward artificial
  neural network.
\newblock {\em Advances in Computational Mathematics 1}, 1 (1993), 61--80.

\bibitem{mhaskar2016}
{\sc Mhaskar, H.~N., and Poggio, T.}
\newblock Deep vs. shallow networks: An approximation theory perspective.
\newblock {\em Analysis and Applications 14}, 06 (2016), 829--848.

\bibitem{Mocanu2018}
{\sc Mocanu, D.~C., Mocanu, E., Stone, P., Nguyen, P.~H., Gibescu, M., and
  Liotta, A.}
\newblock Scalable training of artificial neural networks with adaptive sparse
  connectivity inspired by network science.
\newblock {\em Nature Communications 9}, 1 (2018), 2383.

\bibitem{montufar2013}
{\sc {Mont{\'u}far}, G.~F.}
\newblock Universal approximation depth and errors of narrow belief networks
  with discrete units.
\newblock {\em Neural Computation 26}, 7 (2014), 1386--1407.

\bibitem{pascanu2013}
{\sc {Mont{\'u}far}, G.~F., Pascanu, R., Cho, K., and Bengio, Y.}
\newblock On the number of linear regions of deep neural networks.
\newblock In {\em Advances in Neural Information Processing Systems 27}. Curran
  Associates, Inc., 2014, pp.~2924--2932.

\bibitem{2019arXiv190702177N}
{\sc {Nakada}, R., and {Imaizumi}, M.}
\newblock Adaptive approximation and estimation of deep neural network with
  intrinsic dimensionality.
\newblock {\em arXiv e-prints\/} (Jul 2019), arXiv:1907.02177.

\bibitem{Pedamonti2018}
{\sc {Pedamonti}, D.}
\newblock {Comparison of non-linear activation functions for deep neural
  networks on MNIST classification task}.
\newblock {\em arXiv e-prints\/} (2018), arXiv:1804.02763.

\bibitem{pinkus1999}
{\sc Pinkus, A.}
\newblock {Approximation theory of the MLP model in neural networks}.
\newblock {\em Acta Numerica\/} (1999), 143--195.

\bibitem{Poggio2017}
{\sc Poggio, T., Mhaskar, H., Rosasco, L., Miranda, B., and Liao, Q.}
\newblock Why and when can deep-but not shallow-networks avoid the curse of
  dimensionality: A review.
\newblock {\em International Journal of Automation and Computing 14}, 5 (2017),
  503--519.

\bibitem{Prabhu2018}
{\sc Prabhu, A., Varma, G., and Namboodiri, A.}
\newblock Deep expander networks: Efficient deep networks from graph theory.
\newblock In {\em Computer Vision -- ECCV 2018\/} (Cham, 2018), V.~Ferrari,
  M.~Hebert, C.~Sminchisescu, and Y.~Weiss, Eds., Springer International
  Publishing, pp.~20--36.

\bibitem{ray2017}
{\sc Ray, K., and Schmidt-Hieber, J.}
\newblock A regularity class for the roots of nonnegative functions.
\newblock {\em Annali di Matematica Pura ed Applicata\/} (2017), 1--13.

\bibitem{2018arXiv180905242R}
{\sc {Robinett}, R.~A., and {Kepner}, J.}
\newblock Neural network topologies for sparse training.
\newblock {\em arXiv e-prints\/} (2018), arXiv:1809.05242.

\bibitem{2019arXiv190205040S}
{\sc {Savarese}, P., {Evron}, I., {Soudry}, D., and {Srebro}, N.}
\newblock {How do infinite width bounded norm networks look in function space?}
\newblock {\em arXiv e-prints\/} (2019), arXiv:1902.05040.

\bibitem{SCARSELLI199815}
{\sc Scarselli, F., and Tsoi, A.~C.}
\newblock Universal approximation using feedforward neural networks: A survey
  of some existing methods, and some new results.
\newblock {\em Neural Networks 11}, 1 (1998), 15 -- 37.

\bibitem{2019arXiv190800695S}
{\sc {Schmidt-Hieber}, J.}
\newblock {Deep ReLU network approximation of functions on a manifold}.
\newblock {\em arXiv e-prints\/} (Aug 2019), arXiv:1908.00695.

\bibitem{Soudry2019}
{\sc Soudry, D., Hoffer, E., Nacson, M.~S., Gunasekar, S., and Srebro, N.}
\newblock The implicit bias of gradient descent on separable data.
\newblock {\em Journal of Machine Learning Research 19}, 70 (2018), 1--57.

\bibitem{srivastava2014}
{\sc Srivastava, N., Hinton, G., Krizhevsky, A., Sutskever, I., and
  Salakhutdinov, R.}
\newblock Dropout: a simple way to prevent neural networks from overfitting.
\newblock {\em J. Mach. Learn. Res. 15\/} (2014), 1929--1958.

\bibitem{srivastava2015}
{\sc {Srivastava}, R.~K., {Greff}, K., and {Schmidhuber}, J.}
\newblock {Highway Networks}.
\newblock {\em ArXiv e-prints\/} (2015).

\bibitem{stinchcombe1999}
{\sc Stinchcombe, M.}
\newblock Neural network approximation of continuous functionals and continuous
  functions on compactifications.
\newblock {\em Neural Networks 12}, 3 (1999), 467 -- 477.

\bibitem{MR2500924}
{\sc Strohmer, T., and Vershynin, R.}
\newblock A randomized {K}aczmarz algorithm with exponential convergence.
\newblock {\em J. Fourier Anal. Appl. 15}, 2 (2009), 262--278.

\bibitem{suzuki2017}
{\sc Suzuki, T.}
\newblock Fast generalization error bound of deep learning from a kernel
  perspective.
\newblock In {\em Proceedings of the Twenty-First International Conference on
  Artificial Intelligence and Statistics\/} (2018), vol.~84 of {\em Proceedings
  of Machine Learning Research}, PMLR, pp.~1397--1406.

\bibitem{szegedy2016}
{\sc Szegedy, C., Ioffe, S., Vanhoucke, V., and Alemi, A.~A.}
\newblock Inception-v4, inception-resnet and the impact of residual connections
  on learning.
\newblock In {\em ICLR 2016 Workshop\/} (2016).

\bibitem{telgarsky2016}
{\sc Telgarsky, M.}
\newblock Benefits of depth in neural networks.
\newblock In {\em 29th Annual Conference on Learning Theory\/} (2016), vol.~49
  of {\em Proceedings of Machine Learning Research}, PMLR, pp.~1517--1539.

\bibitem{tsybakov2009}
{\sc Tsybakov, A.~B.}
\newblock {\em Introduction to nonparametric estimation}.
\newblock Springer Series in Statistics. Springer, New York, 2009.

\bibitem{vdVaartWellner}
{\sc van~der Vaart, A.~W., and Wellner, J.~A.}
\newblock {\em Weak convergence and empirical processes}.
\newblock Springer Series in Statistics. Springer-Verlag, New York, 1996.

\bibitem{wasserman2006}
{\sc Wasserman, L.}
\newblock {\em All of nonparametric statistics}.
\newblock Springer Texts in Statistics. Springer, New York, 2006.

\bibitem{yarotski2017}
{\sc Yarotsky, D.}
\newblock Error bounds for approximations with deep {ReLU} networks.
\newblock {\em Neural Networks 94\/} (2017), 103 -- 114.

\bibitem{yarotsky18a}
{\sc Yarotsky, D.}
\newblock {Optimal approximation of continuous functions by very deep ReLU
  networks}.
\newblock In {\em Proceedings of the 31st Conference On Learning Theory\/}
  (06--09 Jul 2018), S.~Bubeck, V.~Perchet, and P.~Rigollet, Eds., vol.~75 of
  {\em Proceedings of Machine Learning Research}, PMLR, pp.~639--649.

\bibitem{2016arXiv161103530Z}
{\sc {Zhang}, C., {Bengio}, S., {Hardt}, M., {Recht}, B., and {Vinyals}, O.}
\newblock {Understanding deep learning requires rethinking generalization}.
\newblock {\em arXiv e-prints\/} (2016), arXiv:1611.03530.

\end{thebibliography}

\end{document}